\documentclass[11pt]{amsart}
\usepackage{amscd,amssymb,graphics}

\usepackage{amsfonts}
\usepackage{amsmath}
\usepackage{amsxtra}
\usepackage{latexsym}
\usepackage[mathcal]{eucal}

\usepackage{graphics,colortbl}

\input xy
\xyoption{all}
\usepackage{epsfig}

\usepackage[pdftex,bookmarks,colorlinks,breaklinks]{hyperref}

\newtheorem{thm}{Theorem}[section]

\newtheorem{cor}[thm]{Corollary}
\newtheorem{lem}[thm]{Lemma}
\newtheorem{prop}[thm]{Proposition}

\theoremstyle{definition}
\newtheorem{defin}[thm]{Definition}

\theoremstyle{remark}
\newtheorem{remarks}[thm]{Remarks}
\newtheorem{remark}[thm]{Remark}

\newtheorem{question}[thm]{Question}
\newtheorem{problem}[thm]{Problem}

\newtheorem{ex}[thm]{Example}
\newtheorem{exs}[thm]{Examples}

\numberwithin{equation}{section}

\newcommand{\delete}[1]{} 
\newcommand{\nt}{\noindent}

\newcommand{\sig}{\sigma}

\def\eps{{\varepsilon}}
\newcommand{\sk}{\vskip 0.3cm}

\newcommand{\ov}{\overline}

\newcommand{\ben}{\begin{enumerate}}

\newcommand{\een}{\end{enumerate}}
\newcommand{\bit}{\begin{itemize}}

\newcommand{\eit}{\end{itemize}}

\def\R {{\mathbb R}}
\def\N {{\mathbb N}}
\def\Z {{\mathbb Z}}
\def\Q {{\mathbb Q}}

\def\T {{\mathbb T}}

\def\RUC{{\hbox{RUC\,}^b}}

\def\H{{\mathrm{H}}}

\def\Iso{{\mathrm{Iso}}\,}
\def\Aut{{\mathrm Aut}\,}

\def\Homeo{{\mathrm{Homeo}}\,}

\def\diam{{\mathrm{diam}}}

\def\A{{\mathcal{A}}}
\def\B{{\mathcal{B}}}

\def\F{{\mathcal F}}

\def\PSL {{\mathrm{PSL}}}
\def\PGL {{\mathrm{PGL}}}

\def\Iso{\operatorname{Iso}}
\def\Asp{\operatorname{Asp}}
\def\RUC{\operatorname{RUC}}

\def\SUC{\operatorname{SUC}}

\def\UC{\operatorname{UC}}
\def\WAP{\operatorname{WAP}}
\def\Tame{\operatorname{Tame}}
\def\AP{\operatorname{AP}}

\def\QED{\nobreak\quad\ifmmode\roman{Q.E.D.}\else{\rm Q.E.D.}\fi}

\def\a{\alpha}

\def\Om{\Omega}

\def\s{\sigma}

\newcommand{\ga}{\gamma}
\newcommand{\del}{\delta}

\newcommand{\br}{\vspace{4 mm}}

\newcommand{\cls}{{\rm{cls\,}}}

\newcommand{\al}{\alpha}
\newcommand{\ep}{\varepsilon}

\newcommand{\lan}{\langle}
\newcommand{\ran}{\rangle}

\newcommand{\Ucal}{\mathcal{U}}

\newcommand{\OC}{\bar{\mathcal{O}}}

\newcommand{\sgn}{{\rm{sgn}}}

\newcommand{\card}{\rm{card\,}}

\newcommand{\proj}{{\rm{proj\,}}}

\begin{document}

\title[]
{More on tame dynamical systems}

\author[]{Eli Glasner}
\address{Department of Mathematics,
Tel-Aviv University, Ramat Aviv, Israel}
\email{glasner@math.tau.ac.il}
\urladdr{http://www.math.tau.ac.il/$^\sim$glasner}

\author[]{Michael Megrelishvili}
\address{Department of Mathematics,
Bar-Ilan University, 52900 Ramat-Gan, Israel}
\email{megereli@math.biu.ac.il}
\urladdr{http://www.math.biu.ac.il/$^\sim$megereli}

\date{revised on 07/02/23}

\begin{abstract}
In this work, on the one hand, we survey and amplify old results concerning tame dynamical systems and, on the other, prove some new results and exhibit new examples of such systems. 
In particular, we study tame symbolic systems and establish a neat characterization of tame subshifts. We also provide sufficient conditions which ensure that certain coding functions are tame. Finally we discuss examples where certain universal dynamical systems 
associated with some Polish groups are tame.
\end{abstract}
 


\thanks{This research was supported by a grant of Israel Science Foundation (ISF 668/13)}

\thanks{The first named author thanks the Hausdorff Institute at Bonn for the opportunity to participate in the Program ``Universality and Homogeneity" where part of this work was written, November 2013.}

\maketitle

\setcounter{tocdepth}{1}
\tableofcontents

\section{Introduction}

Tame dynamical systems were introduced by A. K\"{o}hler \cite{Ko} in 1995 
and their theory 
developed during the last decade in a series of works by several authors
(see e.g. \cite{Gl-tame,GM1,GM-rose,H,KL,Gl-str,Rom}).  
Recently, connections to other areas of mathematics like: Banach spaces, 
circularly ordered systems, 
substitutions and tilings, quasicrystals, cut and project schemes and even model theory and logic were established 
(see e.g. \cite{Auj,Ibar, Ch-Si} and the survey \cite{GM-survey} for more details).

Recall that for any topological group $G$ and any dynamical $G$-system $X$ (defined by a 
continuous homomorphism $j: G \to \H(X)$ into the group of homeomorphisms of the compact space $X$) 
the corresponding enveloping semigroup $E(X)$ 
was defined by Robert Ellis as the 
pointwise closure of the subgroup $j(G)$ of $\H(X)$ in the product space $X^X$. 
$E(X)$ is a compact right topological semigroup whose algebraic and topological structure often reflects properties of $(G,X)$ like almost periodicity (AP),
weak almost periodicity (WAP), distality,
hereditary nonsensitivity (HNS) and tameness, to mention a few. 
In the domain of symbolic dynamics WAP (and even HNS) systems are necessarily countable, and in these classes minimal tame subshifts are necessarily finite. 
In contrast there are 
many interesting symbolic (both minimal and non-minimal) tame systems which are not HNS. 
Sturmian subshifts is an important class of such systems. 

A metric dynamical $G$-system $X$ is tame if and only if 
every element $p \in E(X)$ of the enveloping semigroup $E(X)$ 
is a limit of a sequence of elements from $G$, \cite{GM1,GMU}, 
if and only if its enveloping semigroup $E(X)$ 
has cardinality at most $2^{\aleph_0}$ \cite{GM1,GM-AffComp}. 
For example, the enveloping semigroup of a Sturmian system has the form 
$E(X)=\T_{\T} \cup \Z$, the union of the ``double-circle" 
$\T_\T$ and $\Z$. 
Thus its cardinality is $2^{\aleph_0}$. 
Another interesting property of a Sturmian system $X$ is that both $X$ and $E(X)$ are 
circularly ordered dynamical systems. As it turns out all circularly ordered systems are tame, \cite{GM-c}.

Another characterization of tameness 
of combinatorial nature, via the notion of independence tuples, is due to Kerr and Li \cite{KL}. 
Finally, the metrizable tame systems are exactly those systems which admit 
a representation on a separable Rosenthal Banach space \cite{GM-rose}
(a Banach space is called \emph{Rosenthal} if it does not contain an isomorphic copy of $l_1$). 
As a by-product of the latter characterization we were able in \cite{GM-rose} to 
show that, e.g. for a Sturmian system, the corresponding representation must take place on a separable Rosenthal space $V$ with 
a non-separable dual, thereby proving the existence of such
a Banach space. The question whether such Banach 
spaces exist was a major open problem until the mid of 70's 
 (the first counterexamples were constructed independently by James, and Lindenstrauss and Stegall). 
For a survey of Banach representation theory for dynamical systems we refer the reader to \cite{GM-survey}. 

\sk
In Sections \ref{s:prelim} and \ref{s:classes} we review and amplify some basic results concerning tame systems, fragmentability, and independent families of functions. 
In section \ref{s:symb} we provide a new 
characterization of tame symbolic dynamical systems, 
Theorem \ref{subshifts}, 
and a combinatorial characterization of tame subsets $D \subset \Z$ 
(i.e., subsets $D$ such that the associated subshift $X_D\subset \{0,1\}^{\Z}$ is tame), 
Theorem \ref{tame subsets in Z}. 
In Section \ref{s:entropy} we briefly review results relating tameness to independence and entropy.

\sk

In Section \ref{s:tame-type} we  
study coding functions that yield tame dynamical systems.
A closely related task is to produce invariant 
families of real valued functions which do not 
contain independent infinite  sequences (\emph{tame families}). 
Theorem \ref{2811} gives some useful sufficient conditions for the tameness of families. 
For instance, as a corollary of this theorem, we show in Theorem \ref{t:tame-typeOften}.3 that if
$G$ is a compact metric group and $f: X \to \R$ is a bounded
function with finitely many discontinuities, then $fG$ is a tame family.  

We also describe some old and new interesting examples of symbolic tame systems.
E.g. in Theorem \ref{sphere} we present a Sturmian-like extension
of a rotation on $\T^d$ whose enveloping semigroup has the form 
$E(X) = \Z \cup (\T^d \times \mathcal F)$, where $\mathcal F$ is the collection of
ordered orthonormal bases for $\R^d$.

\sk 

In Section \ref{s:order} we consider dynamical properties which are related to order preservation.
If $X$ is a compact space equipped 
with some kind of order $\le$, 
then subgroups of $\H_+(X, \le)$, the group of order preserving homeomorphisms of $X$, often have some
special properties. 
E.g. $\H_+(\T)$,
the group of orientation preserving homeomorphisms of the circle $\T$,
 is Rosenthal representable, Theorem \ref{t:CoisWRN}, 
and we observe in Theorem \ref{t:RP} that it is Roelcke precompact. 
The recipe described in Theorem \ref{t:multi} 
yields many tame coding functions for subgroups of $\H(\T)$. 
Considering $\Z^k$ or $\PSL_2(\R)$ as subgroups of $\H(\T)$ we obtain 
in this way tame Sturmian like $\Z^k$ and $\PSL_2(\Z)$ 
dynamical systems.

\br

Every topological group $G$ has a universal minimal system $M(G)$,
a universal minimal tame system $M_t(G)$, which is the largest tame $G$-factor of $M(G)$, and also a universal irreducible affine $G$-system $I\!A(G)$. 
In the final Section \ref{sec,int} we discuss some examples, where $M(G)$ and $I\!A(G)$ are tame. 
When $M(G)$ is tame, so that $M(G) = M_t(G)$, we say that $G$ is intrinsically tame.
Of course every extremely amenable group 
(i.e. a group with trivial $M(G)$) 
is intrinsically tame, and
in Theorem \ref{t:intrinsic} we show that the Polish groups $G=\H_+(\T)$
as well as the groups $\Aut(\mathbf{S}(2))$ and $\Aut(\mathbf{S}(3))$, of automorphisms of the circular directed graphs $\mathbf{S}(2)$ and 
$\mathbf{S}(3)$ respectively, are 
all intrinsically tame (but have nontrivial $M(G)$). 
When the universal system $I\!A(G)$ is tame we say that the group $G$ is 
convexly intrinsically tame.
Trivially every amenable group is convexly intrinsically tame, 
and every intrinsically tame group is convexly intrinsically tame.
We show here that
the group $\H_+(\T)$ is a nonamenable convexly intrinsically tame topological group. 
Also, every semisimple Lie group $G$ with finite center and no compact factors 
(e.g., $SL_n(\R)$) 
 is convexly intrinsically tame (but not intrinsically tame) 
 \footnote{An earlier version of this work is posted on the Arxiv (arXiv:1405.2588)}.

\section{Preliminaries}\label{s:prelim}

By a topological space we mean a Tychonoff (completely regular Hausdorff) space. 
The closure operator in topological spaces
will be denoted by $\cls$.  
A function $f: X \to Y$ is \emph{Baire class 1 function} if the inverse image $f^{-1}(O)$ of every open set $O
\subset Y$ is $F_\sigma$ in $X$, \cite{Kech}. 
For a pair of topological spaces $X$ and $Y$, $C(X,Y)$ is the  
set of continuous functions from $X$ into $Y$. 
We denote by $C(X)$ the Banach 
algebra of {\em bounded} continuous real 
valued functions even when 
$X$ is not necessarily compact.

All semigroups $S$ are assumed to be monoids, 
i.e., semigroups with a neutral element which will be denoted by $e$.  
A (left) {\em action} of $S$ on a space $X$ is a map $\pi : S \times X \to X$ such that $\pi(e,x) = x$ and
$\pi(st,x) = \pi(s,\pi(t,x))$ for every $s, t \in S$ and $x \in X$. We usually write simply 
$sx$ for $\pi(s,x)$.

An $S$-\emph{space} is a topological space $X$ equipped with a 
continuous action $\pi: S \times X \to X$ of a 
topological semigroup $S$ on the space $X$. 
A compact $S$-space $X$ is called a \emph{dynamical $S$-system} and is denoted by $(S,X)$.  
Note that in \cite{GM-rose} and \cite{GM-AffComp} we deal with the more general case of separately continuous  actions. 
We reserve the symbol $G$ for the case where $S$ is a topological group. 
As usual, a continuous map $\a : X \to Y$ between two $S$-systems is called an {\em $S$-map}
or {\em a homomorphism}  
when $\a(sx)=s\a(x)$ for every $(s,x) \in S \times X$.  
For every function $f: X \to \R$ and $s \in S$ denote by $fs$ the composition $f \circ \tilde{s}$. That is, $(fs)(x):=f(sx)$.   

For every $S$-system $X$ we have a monoid homomorphism
$j: S \to C(X,X)$, $j(s)=\tilde{s}$,
where $\tilde{s}: X \to X, x \mapsto sx=\pi(s,x)$ is the {\em $s$-translation} ($s \in S$). 
The action is said to be \emph{effective} (\emph{topologically effective}) if $j$ is an injection 
(respectively, a topological embedding).

The {\em enveloping semigroup} $E(S,X)$ (or just $E(X)$)    
for a compact $S$-system $X$ is defined as the pointwise closure ${\cls}(j(S))$ 
of $\tilde{S}=j(S)$ in $X^X$. Then $E(S,X)$ 
is a right topological compact monoid; i.e.
for each $p \in E(X)$ right multiplication by $p$ is a continuous map.

By a \emph{cascade} on a compact space $X$ we mean a $\Z$-action $\Z \times X \to X$.
When dealing with cascades we usually write $(T,X)$ instead of $(\Z,X)$, 
where $T$ is the $s$-translation $X \to X$ corresponding to $s=1 \in \Z$ ($0$ acts as the identity).

\sk

\subsection{Background on fragmentability and tame families}
\label{s:Ban}

The following definitions provide natural generalizations of the fragmentability concept \cite{JR}. 

\begin{defin} \label{def:fr}
	Let $(X,\tau)$ be a topological space and
	$(Y,\mu)$ a uniform space.
	\ben
	\item \cite{JOPV,me-fr} 
	$X$ is {\em $(\tau,
		\mu)$-fragmented\/} by a 
	(typically, not continuous)
	function $f: X \to Y$ if for every nonempty subset $A$ of $X$ and every $\eps
	\in \mu$ there exists an open subset $O$ of $X$ such that $O \cap
	A$ is nonempty and the set $f(O \cap A)$ is $\eps$-small in $Y$.
	We also say in that case that the function $f$ is {\em
		fragmented\/}. Notation: $f \in {\mathcal F}(X,Y)$, whenever the
	uniformity $\mu$ is understood.
	If $Y=\R$ then we write simply ${\mathcal F}(X)$.  
	
	\item \cite{GM1}
	We say that a {\it family of functions} $F=\{f: (X,\tau) \to
	(Y,\mu) \}$ is {\it fragmented} 
	if condition (1) holds simultaneously for all $f \in F$. That is, 
	$f(O \cap A)$ is $\eps$-small for every $f \in F$.
	\item \cite{GM-rose}
	We say that $F$ is an \emph{eventually fragmented family} 
	if every infinite subfamily $C \subset F$ contains
	an infinite fragmented subfamily $K \subset C$.
	\een
\end{defin}

In Definition \ref{def:fr}.1 when $Y=X, f={id}_X$ and $\mu$ is a
metric uniformity, we retrieve the usual definition of
fragmentability (more precisely, $(\tau,\mu)$-fragmentability) in the sense of Jayne and Rogers \cite{JR}.
Implicitly it already appears in a paper of Namioka and Phelps \cite{NP}. 

\begin{lem} \label{r:fr1} \cite{GM1, GM-rose}
	\ben
	\item 
	It is enough to check the conditions of Definition \ref{def:fr} only for $\ep \in \gamma$ from 
	a {\it subbase} 
	$\gamma$ of $\mu$ and for closed nonempty subsets $A \subset X$. 
	\item If $f: (X,\tau) \to (Y,\mu)$ has 	
	\emph{a point of continuity property} PCP (i.e., for every closed
	nonempty $A \subset X$ the restriction $f_{|A}: A \to Y$ has a continuity point) 
	then it is fragmented. If $(X,\tau)$ is hereditarily Baire (e.g., compact, or Polish) and $(Y,\mu)$ is a pseudometrizable uniform space then $f$ is fragmented if and only if $f$ has PCP. 
	So, in particular, for compact $X$, the set $\F(X)$ is exactly $B'_r(X)$ in the notation of \cite{Tal}. 

	\item
	If $X$ is Polish and $Y$ is a separable metric space then
	$f: X \to Y$ is fragmented iff $f$ is a Baire class 1 function (i.e., the inverse image of every open set is
	$F_\sigma$).
	\item Let $(X,\tau)$ be a separable metrizable space and $(Y,\rho)$ a
	pseudometric space. Suppose that $f: X \to Y$ is a fragmented onto map. Then $Y$ is separable. 
	\een 
\end{lem}

For other properties of fragmented maps and fragmented families
we refer to \cite{N, JOPV, KM, me-fr, Me-nz, GM1, GM-rose, GM-tame}.  
Basic properties and applications of fragmentability in topological dynamics can be found in \cite{GM-rose, GM-survey, GM-tame}.

\sk  
\subsection{Independent sequences of functions}
\label{s:ind}

Let $\{f_n: X \to \R\}_{n \in \N}$ be a uniformly bounded sequence of functions on a \emph{set} $X$. Following Rosenthal \cite{Ro} we say that
this sequence is an \emph{$l_1$-sequence} on $X$ if there exists a real constant $a >0$
such that for all $n \in \N$ and choices of real scalars $c_1, \dots, c_n$ we have
$$
a \cdot \sum_{i=1}^n |c_i| \leq ||\sum_{i=1}^n c_i f_i||_{\infty}.
$$

A Banach space $V$ is said to be {\em Rosenthal} if it does
not contain an isomorphic copy of $l_1$, or equivalently, if $V$ does not contain a sequence which is equivalent to an $l_1$-sequence. 

A Banach space $V$ is an {\em Asplund\/} space if the
dual of every separable Banach subspace is separable. Every Asplund space is Rosenthal and every reflexive space is Asplund.

A sequence $f_n$ of	real valued functions on a set
$X$ is said to be \emph{independent} (see \cite{Ro,Tal}) if
there exist real numbers $a < b$ such that
$$
\bigcap_{n \in P} f_n^{-1}(-\infty,a) \cap  \bigcap_{n \in M} f_n^{-1}(b,\infty) \neq \emptyset
$$
for all finite disjoint subsets $P, M$ of $\N$.

\begin{defin} \label{d:tameF} 
	We say that a bounded family $F$ of real valued (not necessarily, continuous) functions on a set $X$ is {\it tame} if $F$ does not contain an independent sequence.
\end{defin}

Every bounded independent sequence is an $l_1$-sequence \cite{Ro}.  
	The sequence of projections on the Cantor cube $$\{\pi_n: \{0,1\}^{\N} \to \{0,1\}\}_{n  \in \N}$$ 
	and the sequence of Rademacher functions
	$$\{r_n: [0,1] \to \R\}_{n \in \N}, \ \ r_n(x):=\sgn (\sin (2^n \pi x))$$
	both are independent (hence, nontame).  
\sk 
The following useful theorem synthesizes some known results.
It mainly is based on results of Rosenthal and Talagrand. 
The equivalence of (1), (3) and (4) is a part of \cite[Theorem 14.1.7]{Tal}
For the case (1) $\Leftrightarrow$ (2) note that every bounded independent sequence $\{f_n: X \to \R\}_{n \in \N}$ is 
an $l_1$-sequence (in the $\sup$-norm), \cite[Prop. 4]{Ro}. On the other hand, as the proof of \cite[Theorem 1]{Ro} shows, if $\{f_n\}_{n \in \N}$ has no independent subsequence then it has a pointwise convergent subsequence. Bounded pointwise-Cauchy sequences in $C(X)$ (for compact $X$) are weak-Cauchy as it follows by Lebesgue's theorem. Now Rosenthal's dichotomy theorem \cite[Main Theorem]{Ro} asserts that $\{f_n\}$ has no $l_1$-sequence.  
In \cite[Sect. 4]{GM-rose} we show why eventual fragmentability of $F$ can be included in this list (item (5)).  

\begin{thm} \label{f:sub-fr} 
	Let $X$ be a compact space and $F \subset C(X)$ a bounded subset.
	The following conditions are equivalent:
	\begin{enumerate}
		\item
		$F$ does not contain an $l_1$-sequence. 
		\item $F$ is a tame family (does not contain an independent sequence). 
		\item
		Each sequence in $F$ has a pointwise convergent subsequence in $\R^X$.
		\item 
		The pointwise closure ${\cls}(F)$ of $F$ in $\R^X$ consists of fragmented maps,
		that is,
		${\cls}(F) \subset {\mathcal F}(X).$
		\item $F$ is an eventually fragmented family.
	\end{enumerate}
\end{thm}

Let $X$ be a topological space and $F \subset l_{\infty}(X)$ be a norm bounded family. 
Recall that $F$ has Grothendieck's {\em Double Limit Property} (DLP) on $X$ if for every sequence $\{f_n\} \subset F$ and every sequence
$\{x_m\} \subset X$ the limits 
$$\lim_n \lim_m f_n(x_m) \ \ \  \text{and} \ \ \ \lim_m \lim_n f_n(x_m)$$
are equal whenever they both exist.

The following examples are mostly reformulations of known results; 
the details can be found in \cite{GM-rose,GM-tame,Me-Helly}. 

\begin{exs} \label{ex:tame} \ 
	\ben  
	\item A Banach space $V$ is Rosenthal iff 
	every bounded subset $F \subset V$ is tame  (as a family of functions) on every bounded subset 
	$Y  \subset V^*$ of the dual space $V^*$, iff $F$ is eventually fragmented on $Y$.  
	\item A Banach space $V$ is Asplund iff every bounded subset $F \subset V$ is a fragmented family of functions on every bounded subset $Y \subset V^*$. 
	\item A Banach space is reflexive iff 
	every bounded subset $F \subset V$ has DLP on every bounded subset 
	$X \subset V^*$.
	\item  ((DLP) $\Rightarrow$ Tame)  
	Let $F$ be a bounded family of real valued (not necessarily continuous) functions on a set $X$ such that $F$ has 
	DLP. Then $F$ is tame.  
	\item
	The family $\Homeo [0,1]$, of all self homeomorphisms of $[0,1]$, 
	is tame (but does not have the DLP on $[0,1]$). 
	\item Let $X$ be a circularly (e.g., linearly) ordered set. Then any bounded family $F$ of 
	real functions with bounded total variation is tame. 
	\een
\end{exs}

Note that in (1), (2) and (3) the converse statements are true; as it follows from results of \cite{GM-tame} 
every bounded tame (DLP) family $F$ on $X$ can be represented 
on a Rosenthal (Asplund, reflexive) Banach space. 
Recall that a representation of $F$ on 
a Banach space $V$ consists of
a pair $(\nu,\a)$ of bounded maps $\nu: F \to V, \ \alpha: X \to V^*$ (with weak-star continuous $\a$) such that 
$$
f(x)= \langle \nu(f), \a(x) \rangle
\ \ \ \forall \ f \in F, \ \ \forall \ x \in X.
$$
In other words, the following diagram commutes
$$\xymatrix{ F \ar@<-2ex>[d]_{\nu} \times X
	\ar@<2ex>[d]^{\a} \ar[r]  & \R \ar[d]^{id } \\
	V \times V^* \ar[r]  &  \R }
$$


\subsection{More properties of fragmented families}

Here we discuss a general principle:
the fragmentability of a family of continuous maps defined on a compact space is ``countably-determined". The following theorem 
is inspired by results of Namioka and can be deduced, after some reformulations, 
from \cite[Theorems 3.4 and 3.6]{N}. See also \cite[Theorem 2.1]{CNO}.

\begin{thm} \label{t:countDetermined}
	Let $F=\{f_i: X \to Y\}_{i \in I}$ be a
	bounded
	family of \textbf{continuous} maps from a compact (not necessarily metrizable)
	space $(X,\tau)$ into a pseudometric space $(Y,d)$.
	The following conditions are equivalent:
	\ben
	\item
	$F$ is a fragmented family of functions on $X$.
	\item
	Every \emph{countable} subfamily $K$ of $F$ is fragmented.
	\item
	For every countable subfamily $K$ of $F$ 
	the pseudometric space $(X,\rho_{K,d})$ is separable,
	where
	$
	\rho_{K,d}(x_1,x_2):=\sup_{f \in K} d(f(x_1),f(x_2)).
	$
	\een
\end{thm}
\begin{proof}
	(1) $\Rightarrow$ (2) is trivial.
	
	(2) $\Rightarrow$ (3):
	Let $K$ be a countable subfamily of $F$.
	Consider the natural map $$\pi: X \to Y^K, \pi(x)(f):=f(x).$$
	By (2), $K$ is a fragmented family.
	This means  (see \cite[Def. 6.8]{GM1}) that 
	the map $\pi$ is $(\tau,\mu_K)$-fragmented, where $\mu_K$ is the uniformity
	of $d$-uniform convergence on $Y^K:=\{f: K \to (Y,d)\}$.
	Then the map
	$\pi$ is also $(\tau,d_K)$-fragmented, where $d_K$ is the pseudometric on $Y^K$ defined by
	$$d_K(z_1,z_2):=\sup_{f \in K} d(z_1(f), z_2(f)).$$
	Since $d$ is bounded, $d_K(z_1,z_2)$ is finite and $d_K$ is well-defined.
	Denote by $(X_K,\tau_p)$ the subspace $\pi(X) \subset Y^K$ in pointwise topology.
	Since $K \subset C(X)$, the induced map $\pi_0: X \to X_K$ is a continuous
	map onto the compact space $(X_K, \tau_p)$. Denote by $i: (X_K,\tau_p) \to (Y^K,d_K)$ the inclusion map.
	So, $\pi=i \circ \pi_0$, where the map $\pi$ is $(\tau,d_K)$-fragmented.
	This easily implies  
	(see \cite[Lemma 2.3.5]{GM-rose}) that 
	$i$ is $(\tau_p,d_K)$-fragmented. It immediately follows that the identity map
	$id: (X_K,\tau_p) \to (X_K,d_K)$ is $(\tau_p,d_K)$-fragmented.
	
	Since $K$ is countable, $(X_K, \tau_p) \subset Y^K$ is metrizable.
	Therefore, $(X_K, \tau_p)$ is second countable (being a metrizable compactum). Now, since $d_K$ is a pseudometric on $Y^K$,
	and $id: (X_K,\tau_p) \to (X_K,d_K)$ is $(\tau_p,d_K)$-fragmented, we can apply Lemma \ref{r:fr1}.4. 
	It directly implies that the set $X_K$ is a separable subset of $(Y^K, d_K)$.
	This means that $(X,\rho_{K,d})$ is separable.
	
	(3) $\Rightarrow$ (1) :
	Suppose that $F$ is not fragmented.
	Thus, there exists a non-empty closed subset $A \subset X$ and an $\eps >0$ such that
	for each non-empty open subset $O \subset X$ with $O \cap A \neq \emptyset$ there is
	some $f \in F$ such that
	$f(O \cap A)$ is not $\eps$-small in $(Y,d)$.
	Let $V_1$ be an arbitrary non-empty relatively open subset in $A$. There are 
	$a,b \in V_1$ and
	$f_1 \in F$ such that $d(f_1(a),f_1(b)) > \eps$. Since $f_1$ is continuous we can choose
	relatively open subsets $V_2,V_3$ in $A$ 
	with $\cls (V_2 \cup V_3) \subset V_1$ 
	such that $d(f_1(x),f_1(y)) > \eps$ for every $(x,y) \in V_2 \times V_3$.

	By induction we can construct a sequence $\{V_n\}_{n \in \N}$ of non-empty relatively open subsets in $A$ and a sequence $K:=\{f_n\}_{n \in \N}$ in $F$ such that:
	
	\bit
	\item [(i)]
	$V_{2n} \cup V_{2n+1} \subset V_n$
	for each $n \in \N$;
	\item [(ii)] $d(f_n(x),f_n(y)) > \eps$ for every $(x,y) \in V_{2n} \times V_{2n+1}$.
	
	\eit
	
	\sk
	
	We claim that $(X,\rho_{K,d})$  is not separable,
	where $$\rho_{K,d}(x_1,x_2):=\sup_{f \in K} d(f(x_1),f(x_2)).$$
	
	In fact, for each \emph{branch}
	$$
	\a:=V_1 \supset V_{n_1} \supset V_{n_2} \supset \cdots
	$$
	where for each $i, n_{i+1}=2n_i$ or $2n_i+1$, by compactness of $X$ one can choose an element
	$$
	x_{\a} \in \bigcap_{i \in \N} {\cls}(V_{n_i}).
	$$
	If $x=x_\alpha$ and $y=x_\beta$ come from different branches,
	then there is an $n \in \N$ such that
	$x \in {\cls}(V_{2n})$ and
	$y \in {\cls}(V_{2n+1})$  (or vice versa).
	In any case it follows from (ii) and the continuity of $f_n$ that $d(f_n(x),f_n(y)) \geq \eps$,
	hence $\rho_{K,d}(x,y) \geq \eps$. Since there are uncountably many branches we conclude
	that $A$ and hence also $X$ are not $\rho_{K,d}$-separable.
\end{proof}

\begin{defin} \label{d:AspSet}
	\cite{Fa, Me-nz}
	Let $X$ be a compact space and $F \subset C(X)$ a norm bounded family of continuous real valued functions on $X$. Then $F$ is said to be an
	\emph{Asplund family for} $X$ if for every countable subfamily $K$ of $F$
	the pseudometric space $(X,\rho_{K,d})$ is separable, 
	where
	$$
	\rho_{K,d}(x_1,x_2):=\sup_{f \in K} |f(x_1) - f(x_2)|.
	$$
\end{defin}

\begin{cor} \label{c:AspSet}
	Let $X$ be a compact space and $F \subset C(X)$ a norm bounded family of continuous real valued functions on $X$. Then $F$ is fragmented if and only if $F$ is an Asplund family for $X$.
\end{cor}

\begin{thm} \label{c:countDetermined2}
	Let $F=\{f_i: X \to Y\}_{i \in I}$ be a family of continuous maps
	from a compact (not necessarily metrizable) space $(X,\tau)$ into a uniform space $(Y,\mu)$.
	Then $F$ is fragmented if and only if every countable subfamily $A \subset F$ is fragmented.
\end{thm}
\begin{proof} The proof can be reduced to Theorem \ref{t:countDetermined}.
	Every uniform space can be uniformly approximated by pseudometric spaces.
	Using Lemma \ref{r:fr1}.1 
	we can assume that $(Y,\mu)$ is pseudometrizable; i.e.
	there exists a pseudometric $d$ such that $\mathrm{unif}(d)=\mu$.
	Moreover, replacing $d$ by the uniformly equivalent pseudometric $\frac{d}{1+d}$ we can assume that $d \leq 1$.
\end{proof}

\section{Classes of dynamical systems}
\label{s:classes}

\begin{defin} \label{d:tameDS} 
	A compact  dynamical $S$-system $X$ is said to be \emph{tame} if
	one of the following equivalent conditions are satisfied:
	\ben 
	\item 
	for every $f \in C(X)$ the family $fS:=\{fs: s \in S\}$
	 has no independent subsequence. 

	\item every $p \in E(X)$ is a \textit{fragmented map} $X \to X$.  
	\item for every $p \in E(X)$ and every $f \in C(X)$ the composition $fp: X \to \R$ has PCP. 
	\een 
\end{defin}

The following principal result is a dynamical analog
of the Bourgain-Fremlin-Talagrand dichotomy \cite{BFT,TodBook}.

\begin{thm} \label{D-BFT} 
	\cite{GM1} \emph{(A dynamical version of BFT dichotomy)}
	Let $X$ be a compact metric dynamical $S$-system and let $E=E(X)$ be its
	enveloping semigroup. 
	Either
	\begin{enumerate}
		\item
		$E$ is a separable Rosenthal compact space (hence $E$ is Fr\'echet and ${card} \; {E} \leq
		2^{\aleph_0}$); or
		\item
		the compact space $E$ contains a homeomorphic
		copy of $\beta\N$ (hence ${card} \; {E} = 2^{2^{\aleph_0}}$).
	\end{enumerate}
	The first possibility
	holds iff $X$ is a tame $S$-system.
\end{thm}

Thus, a metrizable dynamical system is tame iff
$\card(E(X)) \leq 2^{\aleph_0}$ iff $E(X)$ is a Rosenthal compactum (or a Fr\'echet space).
Moreover, by \cite{GMU} a metric $S$-system is tame iff every $p \in E(X)$
is a Baire class 1 map $p: X \to X$. 

\sk
The class of tame dynamical systems is quite large. 
It is closed under subsystems, products and factors.
Recall that an $S$-dynamical system $X$ is weakly almost periodic ($\mathrm{WAP}$) if and only if every $p \in E(X)$ is a continuous map. 
As every continuous map $X \to X$ is fragmented it follows that every WAP system is tame. 
A metrizable $S$-system $X$ is WAP iff $(S,X)$ is representable on 
a reflexive Banach space. 
The class of hereditarily nonsensitive systems ($\mathrm{HNS}$) is an intermediate class of systems, \cite{GM-survey}. 
The property
HNS admits a reformulation in terms of enveloping semigroup:  
$(S,X)$ is HNS iff $E(S,X)$ (equivalently, $\tilde{S}$) is a fragmented family. Of course, this implies that every $p \in E(X)$ is fragmented. So, indeed, 
 $\mathrm{WAP} \subset \mathrm{HNS} \subset \mathrm{Tame}$.
A metrizable $S$-system $X$ is HNS iff $E(X)$ is metrizable iff $(S,X)$ is Asplund representable 
(RN, in another terminology), \cite{GM1,GMU}.

\subsection{Some classes of functions}
\label{s:compactif}

A \emph{compactification} of $X$ is a continuous map $\ga: X \to Y$  with a dense range where $Y$ is 
compact.  
When $X$ and $Y$ are $S$-spaces and $\ga$ is an $S$-map we say that $\ga$ is 
an \emph{$S$-compactification}. 

A function $f \in C(X)$ on an $S$-space $X$ is said to be Right Uniformly Continuous if the induced right action $C(X) \times S \to C(X)$ is continuous at the points $(f,s)$, where $s \in S$. Notation: $f \in \RUC(X)$. 
If $X$ is a compact $S$-space then $\RUC(X)=C(X)$.  
Note that $f \in \RUC(X)$ if and only if there exists an  
$S$-compactification $\ga: X \to Y$ such that $f= \tilde{f} \circ \ga$ for some $\tilde{f} \in C(Y)$.  
In this case we say that $f$ \emph{comes} from the $S$-compactification $\ga: X \to Y$.

The function $f$ is said to be: a)
\emph{WAP}; b) \emph{Asplund}; c)  \emph{tame}
if  $f$ comes from an $S$-compactification $\ga: X \to Y$
such that $(S,Y)$ is: WAP, HNS or tame respectively.
For the corresponding classes of functions we use the notation:
$\WAP (X)$, $\Asp(X),$ $\mathrm{Tame}(X)$, respectively. Each of these is a
norm closed
$S$-invariant subalgebra of the $S$-invariant algebra $\RUC(X)$ and
$\WAP(X) \subset \Asp(X) \subset \mathrm{Tame}(X).$
For more details see \cite{GM-AffComp,GM-survey}. 
As a particular case we have defined the algebras
$\WAP(S)$, $\Asp(S)$, $\mathrm{Tame}(S)$ corresponding to the left action of $S$ on itself.

The $S$-invariant subalgebra $\Tame(S)$ of $\RUC(S)$ induces an $S$-compactification of $S$ 
which we denote by $S \to S^{\Tame}$. Recall that it is a semigroup compactification of $S$ and 
that $S^{\Tame}$ is a compact right topological semigroup, \cite{GM-AffComp}. 
Similarly, one defines the compactifications $S^{\AP}, S^{\WAP}, S^{\Asp}$. 
Here AP means \textit{almost periodic}. AP compact $G$-systems (for groups $S:=G$) are just equicontinuous systems.

\subsection{Cyclic $S$-compactifications}  
\label{s:cyclic}

Let $X$ be an $S$-space. 
For every $f \in \mathrm{RUC}(X)$ define 
the following pointwise continuous natural $S$-map 
$$
\delta_f: X \to \RUC(S), \ \ \delta_f(x)(g):=f(gx). 
$$
It induces an $S$-compactification $\delta_f: X \to X_f$, where $X_f$ is the pointwise closure of  $\delta_f(X)$ in $\RUC(S)$.  
Denote by $\A_f:=\lan fS \ran$
the smallest $S$-invariant unital Banach subalgebra of $\mathrm{RUC}(X)$ which contains $f$. 
The corresponding Gelfand $S$-compactification is equivalent to $\delta_f: X \to X_f$. 
Let 
 $\widetilde{f}:= \widehat{e}|_{X_f}$, where $\widehat{e}$ is the evaluation at $e$ functional on $\RUC(S)$. Then $f$ comes from the $S$-system $X_f$. 
 Moreover, $\widetilde{f}S$ separates points of $X_f$.
	 
	 We call $\delta_f: X \to X_f$ the  \emph{cyclic compactification} of $X$ (induced by $f$), \cite{BJM,GM1,GM-AffComp}. 
	 
	 \begin{defin} \label{d:cyclic} \cite{GM1,GM-AffComp}
	 	We say that a compact  dynamical $S$-system $X$ is \emph{cyclic} if there exists $f \in C(X)$ such that $(S,X)$
	 	is topologically $S$-isomorphic to the Gelfand space $X_f$ of the $S$-invariant unital subalgebra $\A_f \subset C(X)$ generated by the orbit $fS$.
	 \end{defin}

\begin{lem} \label{l:Prop} 
	Let $\ga: X \to Y$ be an $S$-compactification and $f \in C(X)$. 
	\ben
	\item $f$ comes from $\ga$ (i.e., $f=\bar{f} \circ \ga$ for some $\bar{f} \in C(Y)$) if and only if  there exists a continuous onto $S$-map $q:Y \to X_f$ such that  $\bar{f}=\tilde{f}\circ q$ and the following diagram is commutative  
	\begin{equation*}
	\xymatrix { X \ar[d]_{f}  \ar[dr]^{\pi_f} \ar[r]^{\ga} & Y
		\ar[d]^{q} \\
		\R	& \ar[l]_{\widetilde{f}} X_f }
	\end{equation*}

	\item $q: Y \to X_f$ in (1) is an isomorphism of $S$-compactifications if and only if $\bar{f} S$ separates points of $Y$ (where, as before, $\bar{f}=\tilde{f}\circ q$). 
	\een
\end{lem}
\begin{proof}
	Use Gelfand's description of compactifications in terms of the corresponding algebras and the Stone-Weierstrass Theorem. 
	\end{proof}

\begin{remark} \label{r:cycl-comes}
	Let $X$ be a (not necessarily compact) $S$-space and $f \in \RUC(X)$.
	Then, as was shown in \cite{GM-AffComp}, there exist
	a cyclic $S$-system $X_f$, a continuous $S$-compactification $\pi_f: X \to X_f$,
	and a continuous function $\tilde{f}: X_f \to \R$ such that
	$f=\tilde{f} \circ \pi_f$;
	that is, $f$ comes from the $S$-compactification $\pi_f: X \to X_f$.
	The collection of functions $\tilde{f} S$ separates points of $X_f$.
\end{remark}

\sk

\begin{thm} \label{t:tame-f}
	Let $X$ be a compact
	$S$-space and $f \in C(X)$.
	The following conditions are equivalent:
	\ben
	\item 
	$f \in \Tame(X)$ (i.e. $f$ comes from a tame dynamical system).
	\item 
	$fS$ is a tame family. 
	\item
	${\cls}_p(fS) \subset {\mathcal F}(X)$.  
	\item $fS$ is an eventually fragmented family.
	\item For every countable infinite subset $A \subset S$ there exists a countable infinite subset $A' \subset A$ such that
	the corresponding pseudometric
	$$
	\rho_{f, A'}(x,y):=\sup\{|f(gx)-f(gy)|:  \ \ g \in A'\}
	$$
	on $X$ is separable. 
	\item The cyclic $S$-space $X_f$ is Rosenthal representable (i.e., WRN).
	\een
\end{thm} 
\begin{proof} 
	The equivalence of (1), 
	(2), (3) and (4) follows from Theorem \ref{f:sub-fr}.
	For	(4) $\Leftrightarrow$ (5) use Theorem \ref{t:countDetermined}. 
	For (4) $\Leftrightarrow$ (6) we refer to \cite{GM-rose,GM-tame}.  
\end{proof}

\section{A characterization of tame symbolic systems}\label{s:symb}

\subsection{Symbolic systems and coding functions}

The binary \emph{Bernoulli shift system} is defined as the cascade $(\Omega,\s)$, where $\Omega:=\{0,1\}^{\Z}$.
We have the natural
$\Z$-action on the compact metric space  
$\Omega$ induced by the $\s$-shift:
$$
\Z \times \Omega \to \Omega, \ \ \ \s^m (\omega_i)_{i \in \Z}=(\omega_{i+m})_{i \in \Z} \ \ \ \forall (\omega_i)_{i \in \Z} \in \Omega, \ \ \forall m \in \Z.
$$

More generally, for a
discrete monoid $S$ and a finite alphabet 
$A:=\{0,1, \dots,n\}$
the compact space $\Omega:=A^S$ is an $S$-space under the action
$$ S \times \Omega \to \Omega, \ \ (s \omega)(t)=\omega (ts),\ \omega \in A^S, \ \ s,t \in S.$$
A closed $S$-invariant subset $X \subset A^S$ defines a
subsystem $(S,X)$. Such systems are called {\em subshifts\/} or
{\em symbolic dynamical systems\/}.


\begin{defin} \label{d:tametype1} \
	\begin{enumerate}
		\item
		Let $S \times X \to X$ be an action on a (not necessarily compact) space
		$X$,
		$f: X \to \R$ a bounded (not necessarily continuous)
		function, and $z \in X$. Define a \emph{coding function} as follows:
		$$
		\varphi:=m(f,z) : S \to \R, \ s \mapsto f(sz).
		$$ 
		\item
		When $f(X) \subseteq \{0,1,\dots,d\}$ every such code generates a point transitive subshift $S_{\varphi}$ of $A^S$,
		where $A=\{0,1, \dots, d\}$ and 
		$$
		S_{\varphi} : = \cls_p \{g \varphi : g \in S\} \subset  A^S \ \ \ \ (\text{where} \ g \varphi(t)=\varphi(tg))
		$$ 
		is the pointwise closure of the left $S$-orbit $S \varphi$ in the space $\{0,1, \cdots, d\}^S$.
		
		When $S=\Z^k$ we say that $f$ is a \emph{$(k,d)$-code}.
		In the particular case of the characteristic function $\chi_D: X \to \{0,1\}$
		for a subset $D \subset X$ and $S=\Z$ we get a $(1,1)$-code, i.e.
		a binary function $m(D,z): \Z \to \{0,1\}$
		which generates a $\Z$-subshift of the Bernoulli shift on $\{0,1\}^{\Z}$.
	\end{enumerate}
\end{defin}

\sk 

Regarding some dynamical and combinatorial aspects of coding functions 
see \cite{Fernique,BFZ}. 

Among others we will study the following question 

\begin{question} \label{q:coding} 
	When is a coding $\varphi$ function tame? 
	Equivalently, 
	when is the associated transitive subshift system $S_{\varphi} \subset \{0,1\}^{\Z}$ with $\varphi =	m(D,z)$  tame? 
\end{question}

Some restrictions on $D$ are really necessary because \emph{every} binary bisequence 
$\varphi: \Z \to \{0,1\}$ can be encoded as $\varphi=	m(D,z)$. 

It follows from results in \cite{GM-rose} that a coding bisequence
$c: \Z \to \R$ is tame iff it can be represented
as a generalized matrix coefficient of a Rosenthal Banach space representation.
That is, iff there exist: a Rosenthal Banach space $V$, a linear isometry $\s \in \Iso(V)$ and
two vectors $v \in V$, $\varphi \in V^*$ such that
$$
c_n=\langle \s^n(v),\varphi \rangle = \varphi(\s^n(v)) \ \ \ \ \forall n \in \Z.
$$

\sk 

Let, as above, $A^S$ be the full symbolic shift $S$-system.  
For a nonempty $L \subseteq S$ define the natural projection $$\pi_L: A^S \to A^{L}.$$
The compact zero-dimensional space $A^S$ is metrizable iff $S$ is countable 
(and, in this case, $A^S$ is homeomorphic to the Cantor set).

It is easy to see
that the full shift system $\Omega=A^S$ (hence also every subshift) is  \emph{uniformly expansive}. This means that
there exists an entourage $\eps_0 \in \mu$ in the natural uniform structure of $A^S$ such that
for every distinct $\omega_1 \neq \omega_2$ in $\Omega$ one can find
$s \in S$ with
$(s\omega_1, s\omega_2) \notin \eps_0$. Indeed, take
$$\eps_0:=\{(u,v) \in \Omega \times \Omega: \  u(e)=v(e)\},$$
where $e$, as usual, is the neutral element of $S$.

\begin{lem} \label{subshiftsAREcyclic}
Every symbolic  dynamical $S$-system $X \subset \Omega=A^S$ is cyclic (Definition \ref{d:cyclic}).
\end{lem}
\begin{proof} It suffices to find $f \in C(X)$ such that the orbit $fS$ separates the points of $X$ since then, by the Stone-Weierstrass theorem, $(S,X)$ is isomorphic to its cyclic $S$-factor $(S,X_f)$.
The family
$$\{\pi_{s}: X \to A=\{0,1, \dots,n\} \subset \R\}_{s \in S}$$
of basic projections clearly separates points on $X$ and we let
$f:=\pi_e: X \to \R$.
Now observe that $fS=\{\pi_{s}\}_{s \in S}$.
\end{proof}

A topological space $(X,\tau)$ is {\em scattered\/} (i.e., every
nonempty subspace has an isolated point) iff $X$ is
$(\tau,\xi)$-fragmented, for arbitrary uniform structure $\xi$ on the \emph{set} $X$.

\begin{prop} \label{p:scattered} \cite[Prop. 7.15]{Me-nz}
Every scattered compact jointly continuous $S$-space $X$ is RN (that is, Asplund representable). 
\end{prop}
\begin{proof}
A compactum $X$ is scattered iff $C(X)$ is Asplund, \cite{NP}. Now use the canonical 
$S$-representation  
of $(S,X)$ on the Asplund space $V := C(X)$.
\end{proof}

The following result recovers and extends
\cite[Sect. 10]{GM1} and \cite[Sect. 7]{Me-nz}.

\begin{thm} \label{f}
For a discrete monoid $S$ and a finite alphabet $A$ let $X \subset A^S$ be a subshift.
The following conditions are equivalent:
\begin{enumerate}
\item
$(S,X)$ is Asplund representable (that is, RN).
\item $(S,X)$ is HNS.
\item $X$ is scattered.

\sk

\nt If, in addition, $X$ is metrizable (e.g., if $S$ is countable) then
 each of the conditions above is equivalent also to:

\item $X$ is countable.
\end{enumerate}
\end{thm}
\begin{proof}
(1) $\Rightarrow$ (2):  It was proved in \cite[Lemma 9.8]{GM1}. 

 (2) $\Rightarrow$ (3):
Let $\mu$ be the natural uniformity on $X$ and $\mu_S$ the (finer) uniformity of uniform convergence on $X \subset X^S$
(we can treat $X$ as a subset of $X^S$ under the assignment $x \mapsto \hat{x}$, where $\hat{x}(s)=sx$).
If $X$ is HNS then the family $\tilde{S}$ is fragmented. This means that $X$ is $\mu_S$-fragmented.
As we already mentioned, every subshift $X$ is uniformly $S$-expansive.
 Therefore, $\mu_S$ coincides with the discrete uniformity $\mu_{\Delta}$ on $X$ (the largest possible uniformity on the \emph{set} $X$).
 Hence, $X$ is also $\mu_{\Delta}$-fragmented.
This means that $X$ is a scattered compactum. 

 (3) $\Rightarrow$ (1): 
 Use Proposition \ref{p:scattered}.

\sk
If $X$ is metrizable then

 (4) $\Leftrightarrow$ (3): A scattered compactum is metrizable iff it is countable.
\end{proof}

Every zero-dimensional compact $\Z$-system $X$  
can be embedded into a product $\prod X_f$ of (cyclic) subshifts
$X_f$ (where, one may consider only continuous functions $f: X \to \{0,1\}$) of the Bernoulli system $\{0,1\}^{\Z}$.

\sk

For more information about countable
(HNS and WAP) 
subshifts see \cite{Sh, C-W, AkGl-WAP}.

\begin{problem}
Find a nice characterization for WAP (necessarily, countable) $\Z$-subshifts.
\end{problem}

Next we consider tame subshifts.

\begin{thm} \label{subshifts}
Let $X$ be a subshift of $\Omega=A^S$.
The following conditions are equivalent:

\ben
\item
$(S,X)$ is a tame system.
\item
For every infinite subset $L \subseteq S$ there exists an infinite subset
$K \subseteq L$
and a countable subset $Y \subseteq X$ such that
$$
 \pi_{K}(X)=\pi_{K}(Y).
$$
That is,
$$
\forall x=(x_s)_{s \in S} \in X,  \    \exists y=(y_s)_{s \in S} \in Y \quad {\text{with}}\quad
 x_{k}=y_{k} \ \ \forall k \in K.
$$
\item
For every infinite subset $L \subseteq S$ there exists an infinite subset $K \subseteq L$
such that
$\pi_{K}(X)$ is a countable subset of $A^K$.
\item
$(S,X)$ is Rosenthal representable (that is, WRN).
\een
\end{thm}

\begin{proof}
(1) $\Leftrightarrow$ (2): As in the proof of Lemma \ref{subshiftsAREcyclic}
define $f:=\pi_e \in C(X)$. Then $X$ is isomorphic to the cyclic $S$-space $X_f$. 
$(S,X)$ is a tame system iff $C(X)=\Tame(X)$. By Lemma \ref{subshiftsAREcyclic}, $C(X)=\A_f$,
so we have only to show that $f \in \Tame(X)$.

By Theorem \ref{t:tame-f}, $f:=\pi_e: X \to \R$ is a tame function iff
for every infinite subset 
$L \subset S$ there exists a countable infinite subset 
$K \subset L$ such that the corresponding pseudometric
$$
\rho_{f, K}(x,y):=
\sup_{k \in K} \{| (\pi_e)(kx) - (\pi_e)(ky)|\}
= \sup_{k \in K}\{|x_{k} - y_{k}|\}
$$
on $X$ is separable.
The latter assertion means that there exists a countable subset $Y$ which is
$\rho_{f, K}$-dense in $X$.
Thus for every $x \in X$ there is a point $y \in Y$ with $\rho_{f, K}(x,y) < 1/2$.
As the values of the function $f=\pi_{0}$ are in the set $A$, we
conclude that  $\pi_{K}(x)=\pi_{K}(y)$, whence
$$
\pi_{K}(X)=\pi_{K}(Y).
$$

The equivalence of (2) and (3) is obvious.

(1) $\Rightarrow$ (4): $(S,X)$ is Rosenthal-approximable 
(Theorem \ref{t:tame-f}.1).
On the other hand, $(S,X)$ is cyclic (Lemma \ref{subshiftsAREcyclic}).
By Theorem \ref{t:tame-f}.7 we can conclude that $(S,X)$ is WRN.

(4) $\Rightarrow$ (1): Follows directly by Theorem 
\ref{t:tame-f}.1. 
\end{proof}

\begin{remark}
From Theorem \ref{subshifts}  we can deduce the following peculiar fact.
If $X$ is a tame subshift of $\Omega=\{0,1\}^{\Z}$ and $L \subset \Z$ an infinite set,
then there exist an infinite subset $K \subset L$,  $k \ge 1$, and
$a \in \{0,1\}^{2k+1}$ such that
$ X \cap [a] \not=\emptyset$ and
$\forall x, x' \in X \cap [a]$ we have $x|_K = x'|_K$. Here
$[a] =\{z \in \{0,1\}^\Z : z(j) = a(j),\ \forall |j| \le k\}$.
In fact,
since $\pi_K(X)$ is a countable closed
set it contains an isolated point, say $w$, and then the open set
$\pi_K^{-1}(w)$ contains a subset $[a] \cap X$ as required.
\end{remark}

\subsection{Tame and HNS subsets of $\Z$}

We say that a subset $D \subset \Z$ is \emph{tame} if
the characteristic function $\chi_D: \Z \to \R$ is a tame function on the group $\Z$.
That is, when this function \emph{comes}
from a pointed compact tame
$\Z$-system $(X,x_0)$.
Analogously, we say that $D$ is {\em HNS} (or \emph{Asplund}),  \emph{WAP}, or
\emph{Hilbert} if $\chi_D: \Z \to \R$ is an Asplund, WAP or Hilbert function on $\Z$, respectively.
By basic properties of the {\em cyclic system} $X_D : =\cls \{\chi_D \circ T^n : n \in \Z\} \subset \{0,1\}^\Z$
(see Remark \ref{r:cycl-comes}), 
the subset $D \subset \Z$ is tame (Asplund, WAP)
iff the associated subshift $X_D$ is tame (Asplund, WAP).

Surprisingly it is
not known whether
$X_f: =\cls \{f \circ T^n : n \in \Z\} \subset \R^\Z$
is a Hilbert system when $f: \Z \to \R$ is a Hilbert function (see \cite{GW-Hi}).
The following closely related question from \cite{Me-opit} is also open:
Is it true that Hilbert representable compact metric $\Z$-spaces are closed under factors?

\begin{remark} \label{r:Rup}
The definition of WAP sets was introduced by Ruppert \cite{Rup-WAPsets}.
He has the following characterisation (\cite[Theorem 4]{Rup-WAPsets}):
\sk

$D \subset \Z$ is a WAP subset if and only if every infinite subset $B \subset \Z$ contains a finite
subset $F \subset B$ such that the set
$$
\bigcap_{b \in F} (b+D) \setminus \bigcap_{b \in B \setminus F} (b+D)
$$
is finite.
See also \cite{Gl-tf}.
\end{remark}
 
\begin{thm} \label{tame subsets in Z}  
	Let $D$ be a subset of $\Z$. 
The following conditions are equivalent:
\ben
\item
$D$ is a tame subset (i.e., the associated subshift $X_D\subset \{0,1\}^{\Z}$ is tame).
\item
For every infinite subset $L \subseteq \Z$ there exists an infinite subset
$K \subseteq L$
and a countable subset $Y \subseteq \beta \Z$ such that
for every $x \in \beta \Z$ there exists $y \in Y$ such that
$$
n+D \in x \Longleftrightarrow n+D \in y \ \ \ \forall n \in K
$$
(treating $x$ and $y$ as ultrafilters on the set $\Z$). 
\een
\end{thm}
\begin{proof}
By the universality of the greatest ambit $(\Z, \beta \Z)$ it suffices to check when
the function
$$
f=\chi_{\overline{D}}: \beta \Z \to \{0,1\}, \ f(x) =1 \Leftrightarrow x \in \overline{D},
$$
the natural extension function of $\chi_{D}: \Z \to  \{0,1\}$,
is tame (in the usual sense, as a function on the compact cascade $\beta \Z$),
where we denote by $\overline{D}$ the closure of $D$ in $\beta \Z$ (a clopen subset).
Applying Theorem \ref{t:tame-f} to $f$ we see that the following condition is both necessary and sufficient:
For every infinite subset $L \subseteq \Z$ there exists an infinite subset
$K \subseteq L$ and a countable subset $Y \subseteq \beta \Z$ which is dense in the
pseudometric space
$(\beta \Z, \rho_{f, K})$.
Now saying that $Y$ is dense is the same as the requiring that $Y$ be
$\eps$-dense for every $0< \eps < 1$.
However, as $f$ has values in $\{0,1\}$ and $0 <\eps <1$ we conclude that for every $x \in \beta\Z$ there is $y \in Y$ with
$$
x \in n+\overline{D}  \Longleftrightarrow y \in n+\overline{D}  \ \ \ \forall n \in K,
$$
and the latter is equivalent to
$$
n+D \in x \Longleftrightarrow n+D \in y \ \ \ \forall n \in K.
$$
\end{proof}

\begin{thm} \label{Asplund subsets in Z}
		Let $D$ be a subset of $\Z$. 
The following conditions are equivalent:
\ben
\item
$D$ is an Asplund subset (i.e., the associated subshift $X_D\subset \{0,1\}^{\Z}$ is Asplund).
\item
There exists a \emph{countable} subset $Y \subseteq \beta \Z$ such that
for every $x \in \beta \Z$ there exists $y \in Y$ such that
$$
n+D \in x \Longleftrightarrow n+D \in y \ \ \ \forall n \in \Z.
$$

\een
\end{thm}
\begin{proof} (Sketch) One can modify the proof of Theorem \ref{tame subsets in Z}.  Namely, if in assertion (4) of  Theorem \ref{t:tame-f} eventual fragmentability of $F$ is replaced by fragmentability then this characterization of Asplund functions, \cite{GM1} follows.  
\end{proof}

\begin{ex} \label{ex:AspnotWAP}
$\N$ is an Asplund subset of $\Z$ which is not a WAP subset.
In fact, let $X_{\N}$ be the corresponding subshift. Clearly $X_{\N}$ is homeomorphic
to the two-point compactification of $\Z$, with $\{\bf{0}\}$ and $\{\bf{1}\}$ as minimal subsets. Since a transitive WAP
system admits a unique minimal set, we conclude that $X_{\N}$ is not WAP 
(see e.g. \cite{Gl-03}).
On the other hand, since $X_{\N}$ is countable we can apply Theorem \ref{f} to show that it is HNS.
Alternatively, using Theorem \ref{Asplund subsets in Z}, we can take $Y$ to be $\Z \cup \{p,q\}$,
where we choose $p$ and $q$ to be any two non-principal ultrafilters such that $p$ contains $\N$
and $q$ contains $-\N$.
\end{ex}

\section{Entropy and null systems} \label{s:entropy}

We begin by recalling the basic definitions
of topological (sequence) entropy.
Let $(X,T)$ be a cascade,
i.e., a $\Z$-dynamical system, and $A=\{a_0<a_1<\ldots\}$ a
sequence of integers.
Given an open cover $\Ucal$ define
$$
h^A_{top}(T,\Ucal)=\limsup_{n\to\infty} \frac{1}{n}N(\bigvee
_{i=0}^{n-1}T^{-a_i}(\Ucal))
$$
The {\it topological entropy along the sequence $A$} is then defined by
$$
h^A_{top}(T)= \sup
\{h^A_{top}(T,\Ucal) :  \Ucal\ \text{an open cover of $X$}\}.
$$
When the phase space $X$ is zero-dimensional, one can replace
open covers by clopen partitions.
We recall that a dynamical system $(T,X)$ is called {\em null}
if $h^A_{top}(T) =0$ for every infinite $A \subset \Z$.
With $A = \N$ one retrieves the usual definition of topological entropy.
Finally when $Y \subset \{0,1\}^{\Z}$, and $A \subset \Z$
is a given subset of $\Z$, we say that $Y$ {\em is free on} $A$
or that $A$ {\em is an interpolation set for} $Y$,
if $\{y|_A : y \in Y\} = \{0,1\}^A$.

By theorems of Kerr and Li \cite{KL05}, \cite{KL}
every null $\Z$-system is tame, and every tame system has zero topological entropy. 
From results of Glasner-Weiss \cite{GW} (for (1)) and Kerr-Li \cite{KL}
(for (2) and (3)), the following results can be easily deduced.
(See Propositions 3.9.2, 6.4.2 and 5.4.2 of \cite{KL} for the positive topological entropy,
the untame, and the nonnull claims, respectively.)

\begin{thm} \label{t:3}  \
\begin{enumerate}
\item
A subshift $X \subset \{0,1\}^{\Z}$ has positive topological entropy iff
there is a subset $A \subset \Z$ of positive density such that
$X$ is free on $A$.
\item
A subshift $X \subset \{0,1\}^{\Z}$ is not tame iff
there is an infinite subset $A \subset \Z$ such that
$X$ is free on $A$.
\item
A subshift $X \subset \{0,1\}^{\Z}$ is not null iff for every
$n \in \N$ there is a finite subset $A_n \subset \Z$ with $|A_n| \ge n$
such that $X$ is free on $A_n$.
\end{enumerate}
\end{thm}
\begin{proof}
We consider the second claim; the other claims are similar.

Certainly if there is an infinite $A \subset \Z$ on which $X$ is free then
$X$ is not tame (e.g. use Theorem \ref{subshifts}).
Conversely, if $X$ is not tame then, by Propositions 6.4.2 of \cite{KL},
there exists a non diagonal IT pair $(x,y)$.
As $x$ and $y$ are distinct there is
an $n$ with, say, $x(n) =0, y(n)=1$. Since $T^n(x,y)$ is also an IT pair
we can assume that $n=0$. Thus $x \in U_0$ and $y \in U_1$, where these are
the cylinder sets $U_i = \{z \in X : z(0) = i\}, i=0,1$.
Now by the definition of an IT pair there is an infinite set $A \subset \Z$
such that the pair $(U_0,U_1)$ has $A$ as an independence set. This
is exactly the claim that $X$ is free on $A$.
\end{proof}

The following theorem was proved (independently) by
Huang \cite{H}, Kerr and Li \cite{KL}, and Glasner \cite{Gl-str}.
See \cite{Gl-17} for a recent generalization of this result.

\begin{thm} \label{almost-auto}  \emph{(A structure theorem for minimal tame dynamical systems)}
Let $(G,X)$ be a tame minimal metrizable dynamical system with $G$ an
 abelian group.
Then:
\begin{enumerate}
\item
$(G,X)$ is an almost one to one extension $\pi : X \to Y$ of a minimal equicontinuous system $(G,Y)$.
\item
$(G,X)$ is uniquely ergodic and the factor map $\pi$ is, measure
theoretically, an isomorphism of the corresponding measure preserving
system on $X$ with the Haar measure on the equicontinuous factor $Y$.
\end{enumerate}
\end{thm}

\begin{exs} \label{IP} \
\begin{enumerate} 
\item
According to Theorem \ref{almost-auto} the Morse minimal system,
which is uniquely ergodic and has zero entropy,
is nevertheless
 not tame as it fails to be an almost 1-1 extension of its adding machine factor.
We can therefore deduce that, a fortiori, it is not null.
\item
Let $L = IP\{10^t\}_{t=1}^\infty \subset \N$ be the IP-sequence generated by
the powers of ten, i.e.
$$
L =\{10^{a_1} + 10^{a_2} + \cdots + 10^{a_k} :
1 \le a_1 < a_2 < \cdots < a_k\}.
$$
Let $f= 1_L$ and let $X= \OC_\sigma(f) \subset \{0,1\}^\Z$, where $\sigma$ is
the shift on $\Om = \{0,1\}^\Z$.
The subshift $(\sigma,X)$ is not tame. In fact it can be shown that
$L$ is an interpolation set for $X$.
\item
  
Take $u_n$ to be the concatenation of the words $a_{n,i} 0^n$,
where $a_{n,i},\ i=1,2,3, \dots,2^n$ runs over $\{0,1\}^n$. Let $v_n = 0^{|u_n|}$,
$w_n = u_nv_n$ and $w_\infty$ the infinite concatenation
$\{0,1\}^\N \ni w_\infty = w_1w_2w_3\cdots$. Finally define $w \in \{0,1\}^\Z$
by $w(n) = 0$ for $n \le 0$ and $w(n) = w_\infty(n)$.
Then $X=\OC_\sigma(w) \subset \{0,1\}^\Z$ is a countable subshift, hence HNS
and a fortiori tame, but for an appropriately chosen sequence the sequence entropy of $X$ is $\log 2$.
Hence, $X$ is not null. 
Another example of a countable nonnull subshift can be found in \cite[Example 5.12]{H}. 
\item
In \cite[Section 11]{KL} Kerr and Li construct a Toeplitz subshift ( =  a minimal almost one-to-one extension
of an adding machine) which is tame but not null.
\item
In \cite[Theorem 13.9]{GY} the authors show that for interval maps
being tame is the same as being null.
\end{enumerate}
\end{exs}

\begin{remark}
Let $T: [0,1] \to [0,1]$ be a continuous self-map on the closed interval.
In an unpublished paper \cite{Mich+} the authors show that
the enveloping semigroup $E(X)$ of the cascade (an $\N \cup \{0\}$-system) $X=[0,1]$ is either metrizable or
it contains a topological copy of $\beta \N$.
The metrizable enveloping semigroup case occurs exactly when the system is HNS.
This was proved in \cite{GMU} for group actions but it remains true for semigroup actions,
\cite{GM-AffComp}.
The other case occurs iff $\sigma$ is Li-Yorke chaotic.
Combining this result with Example
\ref{IP}.5
one gets:  HNS\ =\ null\ = \ tame, for any cascade
$(T, [0,1])$.
\end{remark}


\section{Some examples of tame functions and systems}
\label{s:tame-type}

In this section we give some methods for constructing tame systems and functions. 
It is closely related to the question whether given family of real (not necessarily, continuous)  functions is tame. 

\sk 
Recall (see for example \cite{BFZ}) that a bisequence $\Z \to \{0,1\}$ is \emph{Sturmian} if it is recurrent and has the minimal complexity $p(n)=n+1$.


\begin{ex} \label{ex:Sturm} \ 
	\ben 
	
	\item (See \cite{GM1}) Consider an irrational rotation $(R_\alpha, \T)$. Choose
	$x_0 \in \T$ and split each point of the orbit $x_n=x_0 + n \alpha$
	into two points  $x_n^{\pm}$.
	This procedure results is a 
	dynamical system $(\s, X)$ which is a minimal almost 1-1 extension of $(R_\alpha,\T)$.
	Then $E(X,\s) \setminus \{\s^n\}_{n \in \Z}$ is homeomorphic to the two arrows space, a basic example of a non-metrizable Rosenthal
	compactum. It follows that $E(\s, X)$ is also a Rosenthal compactum. Hence, $(\s, X)$ is tame but not HNS.  
	\item 
	Let $P_0$ be the set
	$[0, c)$ and $P_1$ the set $[c, 1)$; let $z$ be a point in $[0, 1)$ (identified with $\T$) via the rotation $R_{\a}$ we get the binary bisequence $u_n$, $n \in \Z$
	defined by $u_n = 0$ when $R_{\a}^n(z) \in P_0, u_n = 1$ otherwise.
	These are called \emph{Sturmian like codings}.
	With $c = 1 - \al$ we 
	get the classical Sturmian bisequences.
	For example, when $\a:=\frac{\sqrt{5}-1}{2}$ 
	the corresponding sequence,
	computed at $z =0$, is called the \emph{Fibonacci bisequence}.
	\een 
	\end{ex}

\begin{ex} \label{e:tameNOThns} \
\ben
\item
In his paper \cite{Ellis93} Ellis, following Furstenberg's
classical work \cite{Furst-63-Poisson}, investigates the projective
action of $GL(n,\R)$
on the projective space $\mathbb{P}^{n-1}$. It follows from his
results that the corresponding enveloping semigroup is
not first countable. However,
in a later work \cite{Ak-98}, Akin studies the action of $G=GL(n,\R)$
on the sphere $\mathbb{S}^{n-1}$ and shows that here the enveloping
semigroup is first countable (but not metrizable).
It follows that the dynamical systems
$D_1=(G, \mathbb{P}^{n-1})$ and $D_2=(G, \mathbb{S}^{n-1})$
are tame but not HNS. Note that $E(D_1)$ is Fr\'echet,
being a quotient of a first countable compact space, namely $E(D_2)$.

\item (Huang \cite{H})
An almost 1-1 extension $\pi: X \to Y$ of
an equicontinuous metric $\Z$-system $Y$ with $X \setminus X_0$ countable,
where $X_0=\{x \in X : |\pi^{-1}\pi(x)|=1\}$, is tame.
\een
\end{ex}

We will see that many coding functions are tame, including some
multidimensional analogues of Sturmian sequences.
The latter are defined on the groups $\Z^k$ and instead of the characteristic function
$f:=\chi_D$ (with $D=[0,c)$) one may consider coloring of the space leading to shifts with finite alphabet.
We give a precise definition which (at least in some partial cases)
was examined in several papers. 
Regarding some dynamical and combinatorial aspects of coding functions  
see for example \cite{BV,Fernique,Pikula}, and the survey paper \cite{BFZ}. 

\begin{defin} \label{d:mSturm} 
	Consider an arbitrary finite partition 
	$$\T=\cup_{i=0}^{d} [c_i,c_{i+1})$$ of $\T$ 
	by the ordered $d$-tuple of points $c_0=0,c_1, \dots, c_{d}, c_{d+1}=1$ and  
any coloring map
	$$
	f: \T \to A:=\{0, \dots ,d\}. 
	$$
	Now for a given $k$-tuple
	$(\a_1,\dots ,\a_k) \in \T^k$  and a given point $z \in \T$ consider the corresponding coding function
	$$
	m(f,z): \Z^k \to \{0, \dots ,d\} \ \ \ (n_1, \dots,n_k) \mapsto f(z + n_1 \a_1 + \cdots + n_k \a_k).
	$$
	We call such a sequence a \emph{multidimensional 
		$(k,d)$-Sturmian like sequence}.  
\end{defin}

\begin{lem} \label{l_1} \
\ben
\item
Let $q: X_1 \to X_2$ be a map between sets and
$\{f_n: X_2 \to \R\}_{n \in \N}$ a bounded sequence of functions \emph{(with no continuity assumptions on $q$ and $f_n$)}.
If $\{f_n \circ q\}$ is an independent sequence on $X_1$
then $\{f_n\}$ is an independent sequence on $X_2$.
\item
If $q$ is onto then the converse is also true. That is $\{f_n \circ q\}$ is independent if and only if $\{f_n\}$ is independent.
\item
Let $\{f_n\}$ be a bounded sequence of continuous
functions on a topological space $X$.
Let $Y$ be a \emph{dense} subset of $X$. Then $\{f_n\}$ is an independent sequence on $X$ if and only if
the sequence of restrictions
$\{f_n|_Y\}$ is an independent sequence on $Y$.
\een
\end{lem}
\begin{proof}
Claims (1) and (2) are straightforward.

(3)
 Since $\{f_n\}$ is an independent sequence 
for every pair of finite disjoint sets
$P, M \subset \N$,
the set
$$
\bigcap_{n \in P} f_n^{-1}(-\infty,a) \cap  \bigcap_{n \in M} f_n^{-1}(b,\infty)
$$
is non-empty. 
This set is open because every $f_n$ is continuous. 
Hence, each of them meets the dense set $Y$.
As
$f_n^{-1}(-\infty,a) \cap Y= f_n|_Y^{-1}(-\infty,a)$ and
$f_n^{-1}(b,\infty) \cap Y= f_n|_Y^{-1}(b,\infty)$,
this implies that $\{f_n|_Y\}$ is an independent sequence on $Y$.

Conversely if $\{f_n|_Y\}$ is an independent sequence on a subset $Y \subset X$ then by (1)
(where $q$ is the embedding $Y \hookrightarrow X$),
$\{f_n\}$ is an independent sequence on $X$.
\end{proof}

Below we will sometimes
deal with (not necessarily continuous) functions $f: X \to \R$ such that the orbit $fS$ of $f$ in $\R^X$ is a tame family (Definition \ref{d:tameF}). 
An example of such  Baire 1 function (which is not tame, being discontinuous), 
is the characteristic function $\chi_D$ of an arc $D=[a,a+s)\subset \T$ defined on the system $(R_{\a}, \T)$, where $R_{\a}$
is an irrational rotation of the circle $\T$. 
See Theorem {t:tame-typeOften}. 

\begin{lem} \label{l:tametype}  
Let $S$ be a semigroup, $X$ a (not necessarily compact) $S$-space and
$f: X \to \R$ a bounded (not necessarily continuous) function.
\ben 
\item
Let $f \in \RUC(X)$; then
$f \in \Tame(X)$ if and only if $fS$ is a tame family. 
Moreover, there exists an $S$-compactification $\nu: X \to Y$ 
where the action $S \times Y \to Y$ is continuous,
$Y$ is a tame system and $f=\tilde{f} \circ \nu$ for some $\tilde{f} \in C(Y)$.
\item
Let $G$ be a topological group and $f \in \RUC(G)$. Then $f \in \Tame(G)$ 
if and only if $fG$ is a tame family. 
\item
Let $L$ be a discrete semigroup and $f: L \to \R$ a bounded function. Then $f \in \Tame(L)$ if and only if $fL$ is a tame family. 
\item
Let $h: L \to S$ be a homomorphism of semigroups,  
$S \times Y \to Y$ be an action (without any continuity assumptions) on a set $Y$ and $f: Y \to \R$ be 
a bounded function such that $fL$ is a tame family. 
Then for every point $y \in Y$ the corresponding coding function $m(f,y): L \to \R$ is tame on the discrete semigroup $(L,\tau_{discr})$.
\een
\end{lem}
\begin{proof} 
For (1) consider the cyclic $S$-compactification $\nu: X \to Y=X_f$
(see Definition \ref{d:cyclic}). Since $f \in \RUC(X)$ the action $S \times X_f \to X_f$ is jointly continuous (Remark \ref{r:cycl-comes}).
By the basic property of the cyclic compactification there exists a continuous function $\tilde{f}: X_f \to \R$ 
such that $f=\tilde{f} \circ \nu$. The family $fS$ has no independent sequence.
By Lemma \ref{l_1}.3 we conclude that also $\tilde{f} S$ has no independent sequence. 
This means, by Theorem \ref{t:tame-f}, that $\tilde{f}$ is tame. Hence (by Definition) so is $f$.
The converse follows from Lemma \ref{l_1}.1.

(2) and (3) follow easily from (1) (with $X=G=L$) taking into account that on a
discrete semigroup $L$ every bounded function $L \to \R$ is in $\RUC(L)$.

(4)
 By (3) it is enough to show for 
the coding function $f_0:=m(f,y)$ that the family $f_0L$ has no independent subsequence.
Define $q: L \to Y, s \mapsto h(s)y$.
Then $f_0t=(ft) \circ q$ for every $t \in L$. If $f_0t_n$ is an independent sequence for some sequence $t_n \in L$ then Lemma \ref{l_1}.1 
implies that the sequence of functions $ft_n$ on $Y$ is independent.
This contradicts the assumption that $fL$ has no independent subsequence.
\end{proof}

\sk

Let $f: X \to Y$ be a function between topological spaces. We denote by $cont(f)$ and $disc(f)$ the points of continuity and discontinuity for $f$
respectively.

\begin{defin} \label{d:EvCont}
	Let $F$ be a family of functions on $X$. We say that $F$ is:
	\ben
	\item \emph{Strongly almost continuous}
	if for every $x \in X$
	we have $x \in cont(f)$  for
	almost all $f \in F$ (i.e. with the exception of at most a finite set of elements
	which may depend on $x$).
	\item
	\emph{Almost continuous}
	if for every infinite (countable) subset $F_1 \subset F$ there exists an
	infinite subset $F_2 \subset F_1$ such that $F_2$ is strongly almost continuous on $X$.
	\een
\end{defin}

\begin{ex} \label{ex:event} \
	\ben
	\item Let $G \times X \to X$ be a group action, $G_0 \leq G$
	a
	subgroup and $f: X \to \R$
	a function such that
	$$
	G_0x \cap disc(f) \ \text{and} \ St(x) \cap G_0 \ \  \text{are \ finite}  \ \forall x \in X,
	$$
	where $St(x) \leq G$ is the stabilizer subgroup of $x$.
Then the family $fG_0$ is strongly almost continuous
(indeed, use the following equality $g^{-1}cont(f)=cont(fg), \ g \in G$).
\item
A coarse sufficient condition for (1) is:
$disc(f)$ is finite and $St(x) \cap G_0$ is finite \ $\forall x \in X$. 
\item As a particular case of (2)
we have the following example.
For every compact group $G$ and a function $f: G \to \R$ with finitely many discontinuities,
$fG_0$ is strongly almost continuous on $X=G$ for every subgroup $G_0$ of $G$.
\een
\end{ex}

\begin{thm} \label{2811}
	Let $X$ be a compact metric space and
	$F$ a bounded family of real valued functions on $X$ such that $F$ is almost continuous.
	Further assume that:
	\bit
	\item [(*)]
	for every sequence $\{f_n\}_{n \in \N}$ in $F$ there exists a subsequence $\{f_{n_m}\}_{m \in \N}$ and a countable
	subset $C \subset X$ such that
	$\{f_{n_m}\}_{m \in \N}$ pointwise converges on $X \setminus C$
	to a function $\phi: X \setminus C \to \R$ where $\phi \in \B_1(X \setminus C)$.
	\eit
	Then $F$ is a tame family. 
\end{thm}

%

%
%
%
\begin{proof} 	
	Assuming the contrary let $\{f_n\}$ be an independent sequence in $F$. Then,
	by assumption,
	there exists a countable subset $C \subset X$ and a subsequence $\{f_{n_m}\}$ such that
	$\{f_{n_m}: X \setminus C \to \R\}$ pointwise converges on $X \setminus C$
	to a function $\phi: X \setminus C \to \R$ such that $\phi \in \B_1(X \setminus C)$.
	
	Independence is preserved by subsequences so
	this subsequence $\{f_{n_m}\}$ remains independent. For simplicity of notation assume that $\{f_n\}$ itself has the properties of $\{f_{n_m}\}$.
	Moreover we can suppose in addition,
	by Definition \ref{d:EvCont}, that $\{f_n\}$ is strongly almost continuous.

By the definition of independence, there exist $a < b$ such that for every pair of disjoint finite sets
$P,M \subset \N$ we have 
$$
\bigcap_{n \in P} A_n \cap  \bigcap_{n \in M} B_n \neq \emptyset,
$$
where $A_n:=f_n^{-1}(-\infty,a)$ and $B_n:=f_n^{-1}(b,\infty)$.
Now define
a tree of nested sets
as follows:

\hskip 6.5cm $\Omega_1:=X$

\hskip 2cm $\Omega_2:=\Omega_1 \cap A_1=A_1$ \hskip 3cm  $\Omega_3:=\Omega_1 \cap B_1=B_1$

$\Omega_4:=\Omega_2 \cap A_2  \hskip 1.2cm \Omega_5:=\Omega_2 \cap B_2 \hskip 1.2cm \Omega_6:=\Omega_3\cap A_2 \hskip 1cm \Omega_7:=\Omega_3 \cap B_2$,

and so on.
In general, $$\Om_{2^{n+1} + 2k}:=\Om_{2^n+k} \cap A_{n+1},  \hskip 0.3cm  \Om_{2^{n+1} + 2k+1}:=\Om_{2^n+k} \cap B_{n+1}$$
for every $0 \leq k < 2^n$ and every $n \in \N$.

We obtain a system $\{\Om_n\}_{n \in \N}$ which satisfies:

$$
\Om_{2n} \cup \Om_{2n+1} \subset \Om_n
\ {\text{ and}}\   \Om_{2n} \cap \Om_{2n+1} =\emptyset
\  {\text{for each}}\  n \in \N.
$$

\sk

Since $\{(A_n,B_n)\}_{n \in \N}$ is independent 
(in the sense of \cite{Ro}), 
every $\Om_n$ is nonempty.

For every binary sequence $u=(u_1,u_2, \dots) \in \{0,1\}^{\N}$ we have the corresponding uniquely defined
\emph{branch}
$$
\a_u:=\Om_1 \supset \Om_{n_1} \supset \Om_{n_2} \supset \cdots
$$
where for each $i \in \N$ with $2^{i-1} \leq n_i < 2^{i}$ we have
$$n_{i+1}=2n_i \ \text{iff} \ u_i=0 \ \text{and} \ n_{i+1}=2n_i+1 \ \text{iff} \ u_i=1.$$

Let us say that $u, v \in \{0,1\}^{\N}$ are \emph{essentially distinct} if they have infinitely many different coordinates.
Equivalently, if $u$ and $v$ are in different cosets of the Cantor group $\{0,1\}^{\N}$  with respect to the subgroup $H$ consisting of
the binary sequences with finite support. Since $H$ is countable there are uncountably many pairwise essentially distinct
elements in the Cantor group. We choose a subset $T \subset \{0,1\}^{\N}$ which intersects each coset in exactly one point.  Clearly, 
$card (T)=2^{\omega}$. Now 
for every branch $\a_u$ where $u \in T$
choose one element
$$x_{u} \in \bigcap_{i \in \N} cl(\Om_{n_i}).$$
Here we use the compactness of $X$ which guarantees that $\bigcap_{i \in \N} cl(\Om_{n_i}) \neq \emptyset$.
We obtain a set $X_T:=\{x_{u}: \ u \in T\} \subset X$ and an onto function $T \to X_T, \ u \mapsto x_u$. 

\sk

\textbf{Claim:}
\ben
\item  The function $T \to X_T, \ u \mapsto x_u$ is injective. In particular,
$X_T \setminus C$ is uncountable.
\item
$|\phi(x_u) - \phi(x_v)| \geq \eps:=b-a$ for every
distinct $x_u, x_v \in X_T \setminus C$.
\een

\sk
\textbf{Proof of the Claim:} (1)
Let $u=(u_i)$ and $v=(v_i)$ are distinct elements in $T$. Denote by
$\a_u:=\{\Omega_{n_i}\}_{i \in \N}$ and $\a_v:=\{\Omega_{m_i}\}_{i \in \N}$ the corresponding branches.
Then, by the definition of $X_T$, we have the uniquely defined points
$x_u \in \cap_{i \in \N} cl(\Om_{n_i})$ and
$x_v \in \cap_{i \in \N} cl(\Om_{m_i})$ in $X_T$.

Since $u,v \in T$ are essentially distinct they have infinitely many different indices.

As $\{f_n\}$ is strongly almost continuous
there exists a sufficiently large $t_0 \in \N$ such that
the points $x_u$ and $x_v$ are both points of continuity of
$f_{n}$ for every $n \geq t_0$.

Now note that if
$u_{i} \neq v_{i}$ then
the sets $\Omega_{n_{i}+1}$ and $\Omega_{m_{i}+1}$ are contained
(respectively) in the
pair of disjoint sets
$A_k:=f_k^{-1}(-\infty,a)$ and $B_k:=f_k^{-1}(b,\infty)$.
Since $u$ and $v$ are essentially distinct we can assume that $i$ is sufficiently large in order to ensure that $k \geq t_0$.
That is, we necessarily have exactly one of the cases:
$$(a) \ \  \ \Omega_{n_{i}+1} \subset A_k, \quad \Omega_{m_{i}+1} \subset B_k$$ or
$$(b) \ \ \  \Omega_{n_{i}+1} \subset B_k, \quad \Omega_{m_{i}+1} \subset A_k.$$

For simplicity we only check the first case (a). For (a) we have
$x_u \in \cls(\Omega_{n_{i}+1}) \subset \cls(f_{k}^{-1}(-\infty,a))$ and $x_v \in \cls(\Omega_{n_{i}+1}) \subset \cls(f_{k}^{-1}(b,\infty))$.
Since $\{x_u,x_v\} \subset cont(f_{n})$ are continuity points for every $n \geq t_0$ and since $k \geq t_0$ by our choice,
we obtain $f_{k}(x_u) \leq a$ and $f_{k}(x_v) \geq b$. So, we can conclude
that $|f_{k}(x_u) -f_{k}(x_v)| \geq \eps:=b-a$ for every $k \geq t_0$. In particular, $x_u$ and $x_v$ are distinct.
This proves (1).

(2) 
Furthermore, if our distinct $x_u, x_v \in X_T$ are in addition from $X_T \setminus C$ then by (*) we have 
$\lim f_k(x_u)=\phi(x_u)$ and $\lim f_k(x_v)=\phi(x_v)$. It follows that
$|\phi(x_u) - \phi(x_v)| \geq \eps$ and
the condition (2) of our claim is also proved.

\sk
Since $X_T \setminus C$ is an uncountable subset of a Polish space $X$ there exists an uncountable subset $Y \subset X_T \setminus C$ such that any point of $y$ is a condensation point in $X_T \setminus C$.  This follows from the proof of Cantor-Bendixson theorem, 
\cite{Kech}). More precisely, define 
$$Q:=\{x \in X_T \setminus C: 
\ \text{there exists a countable open nbd} \ O_x \ \text{of} \ x \ \text{in the space} \ X_T \setminus C \}.$$
Observe that $Q= \bigcup \{ O_x: \ x \in Q \}$. Since $Q$ is second countable, by Lindelof property there exists a countable subcover. Hence, $Q$ is at most a countable subset of $X_T \setminus C$ and any point $y \in Y:= (X_T \setminus C) \setminus Q$ is a condensation point. 

Now, it follows by assertion (2) of the Claim that  
for every open subset $U$ in $X$ with
$U \cap Y \neq \emptyset$ we have $\diam(\phi(U \cap Y)) \geq \eps$. This
means that $\phi: X \setminus C \to \R$ is not fragmented. Since $C$ is countable and $X$ is compact metrizable the subset $X \setminus C$ is
Polish. On Polish spaces fragmentability and Baire 1 property are the same for real valued functions (Lemma \ref{r:fr1}.2).
So, we obtain that $\phi: X \setminus C \to \R$ is not Baire 1. This contradicts
the assumption that $\phi \in \B_1(X \setminus C)$.


\end{proof}

%
%
%

\begin{thm} \label{c:2811}
	Let $X$ be a compact metric space,
	and $F$ a bounded family of real valued functions on $X$ such that $F$ is almost continuous.
	Assume that $\cls_p(F) \subset \B_1(X)$.
	Then $F$ is a tame family. 
\end{thm}
\begin{proof} Assuming the contrary let $F$ has an independent sequence $F_1:=\{f_n\}$.
	Since $\cls_p(F) \subset \B_1(X)$ we have $\cls_p(F_1) \subset \B_1(X)$.
	By the BFT theorem \cite[Theorem 3F]{BFT} the compactum $\cls_p(F_1)$ is a Fr\'echet topological space.
	Every (countably) compact Fr\'echet
	space is sequentially compact, \cite[Theorem 3.10.31]{Eng},
	hence
	$\cls_p(F_1)$ is sequentially compact.
	Therefore the sequence $\{f_n\}$
	contains a pointwise
	convergent subsequence, say $f_n \to \phi \in \B_1(X)$.
	Now apply Theorem \ref{2811}
	to get a contradiction, taking into account that
	the properties almost continuity and independence are both inherited
	by subsequences.
\end{proof}


\begin{thm} \label{t:tame-typeOften} \
	\ben
	\item Let $X$ be a compact metric $S$-space, $S_0$
	a subsemigroup of $S$.  Let $f: X \to \R$
	be a bounded function such that $\cls_p(fS_0) \subset \B_1(X)$
	and  with $fS_0$ almost continuous. 
	Then $fS_0$ is 
	a tame family. 
	
	\item
	Let $X$ be a compact metric $G$-space
	and $G_0 \leq G$
	a subgroup of $G$ such that (i) $(G_0,X)$ is tame, (ii)
	for every $p \in E(G,X)$ and every $x \in X$ the preimage $p^{-1}(x)$ is countable
	and (iii) $G_0 \cap St(x)$ is finite for every $x \in X$.
	Suppose further that $f: X \to \R$ is a bounded function with only finitely many points of discontinuity. 
	Then $fG_0$ is
	a  tame family. 
	\item
	In particular, (2) holds in the following useful situation: $X=G$ is a compact metric group,
	$h: G_0 \to X$ is a homomorphism of groups and $f: X \to \R$ has finitely many discontinuities. 
	Then $fG_0$ is 
	a tame family.
	
	\item
	In all the cases above (1), (2), (3), a coding function $m(f,z): S_0 \to \R$ is a tame function
	on the discrete semigroup $S_0$ for every $z \in X$ and every homomorphism
	$h: S_0 \to S$ (or, $G_0 \to G$). Also the corresponding subshift $(S_0, X_f)$ is tame.
	
	\een
\end{thm}
\begin{proof} (1) Apply Theorem \ref{c:2811}.
	
	(2)
	First note that $fG_0$ is strongly almost continuous (Example \ref{ex:event}.2).
	Now assuming the contrary $fG_0$ has an independent subsequence $\{fg_n\}$. Since
	$(G_0, X)$ is a tame metric system we can assume, with no loss in generality,
	that the sequence $\{g_n\}$ converges to an element $p \in E(G_0,X)$.
	Then $f(g_n(x))$ converges to $f(px)$ for every $x \in X \setminus C$, where
	$C: = p^{-1}( disc(f))$ is a countable set.
	Since $X$ is a tame system, $p: X \to X$ is a fragmented function.
	Then also the restricted function $p_0: X \setminus C \to X$ is fragmented.
	Since $f: X \to \R$ is uniformly continuous we obtain that the composition $f \circ p_0: X \setminus C \to \R$ is fragmented.
	Since $C$ is countable, $X \setminus C$ is Polish. Therefore, by Lemma \ref{r:fr1}.2,
	$f \circ p_0: X \setminus C \to \R$ is Baire 1.
	This however is in contradiction with Theorem \ref{2811}.
	
	(4) Follows from (1) and Lemma \ref{l:tametype}. 
\end{proof}

Theorem \ref{t:tame-typeOften}.4 directly implies the following:

\begin{ex}
	For every irrational rotation $\alpha$ of the circle $\T$ and an
	arc \ $D:=[a,b) \subset \T$
	the function $$\varphi_D:=\Z \to \R, \quad n \mapsto \chi_D(n \alpha)$$
	is a tame function on the group $\Z$.
	In particular, for $D:=[-\frac{1}{4},\frac{1}{4})$ we get that $\varphi_D(n)=\sgn \cos (2 \pi n \a)$ is a tame function on $\Z$.
\end{ex}

\begin{thm} \label{multi}
	The multidimensional Sturmian $(k,d)$-sequences
	$\Z^k \to \{0,1, \dots, d\}$ (Definition \ref{d:mSturm}) are tame.
\end{thm}
\begin{proof} In terms of Definition \ref{d:mSturm} consider the homomorphism
	$$h: \Z^k \to \T, \ \ \ (n_1, \dots,n_k) \mapsto n_1 \a_1 + \dots + n_k \a_k.$$
	The function $f$ induced by a given partition $\T=\cup_{i=0}^{d} [c_i,c_{i+1})$
	$$
	f: \T \to A:=\{0, \dots, d\}, \ \ \ f(t)=i \ \text{iff} \ t \in [c_i,c_{i+1}).
	$$
	has only finitely many discontinuities.
	Now Theorem \ref{t:tame-typeOften} (items (3) and (4)) guarantees that the corresponding
	$(k,d)$-coding function $m(f,z): \Z^k \to A \subset \R$ is tame for every $z \in \T$.
\end{proof}

Note that the tameness of the 
functions on $\Z^k$ from Theorem \ref{multi} 
follows also by results from \cite{GM-c}. 
In fact, such functions come from circularly ordered metric dynamical $\Z^k$-systems. See also Theorem \ref{t:multi} and Remark \ref{r:nonc} below. At least for $(1,d)$-codes, Theorem \ref{multi}, can be derived also from results of Pikula \cite{Pikula},
and 
of Aujogue \cite{Auj}.

\sk
\subsection{Strong almost 1-1 equivalence and tameness}

		Recall that a $G$-factor $\pi: X \to Y$ 
		is said to be an \emph{almost one-to-one extension} if $$X_0:=\{x \in X: \ \ |\pi^{-1}(\pi(x))|=1\}$$ 
		is a residual subset of $X$. 
		We will say that $\pi : X \to Y$ is a \emph{strongly almost 1-1 extension} if $X \setminus X_0$ is at most countable. 
		


We say that compact dynamical $G$-systems $X, Y$ are \emph{strongly almost 1-1 equivalent}  
if there exist a continuous $G$-map $\pi: X \to Y$ and two countable subsets $X_1 \subset X, Y_1 \subset Y$ such that the restriction $\pi: X \setminus X_1 \to Y \setminus Y_1$ is bijective.   
One may show that a surjective strongly almost 1-1 equivalence $\pi: X \to Y$ is exactly a strongly almost 1-1 extension.

\begin{remark} \label{r:semi-conj} 
In \cite[p. 30]{diss:Jolivet} Jolivet
calls a strong almost 1-1 equivalence ``semi-conjugation". However, the name semi-conjugation
is often used as a synonym to factor map; so we use ``strong almost 1-1 equivalence" instead.
\end{remark}

The next lemma is well known; for completeness we provide the short proof.

\begin{lem} \label{claim}
Let
$\pi: X \to Y$ be a continuous onto $G$-map of compact metric $G$-systems. Set
$$
X_0:=\{x \in X: \ \ |\pi^{-1}(\pi(x))|=1\}.
$$
Then the restriction map $\pi: X_0 \to Y_0$ is a topological homeomorphism of $G$-subspaces,
where $Y_0:=\pi(X_0)$.
\end{lem}
\begin{proof} Since $G$ is a group 
$X_0$ and $Y_0$ are $G$-invariant. 
The map $\pi: X_0 \to Y_0$ is a continuous bijection.
For every converging sequence $y_n \to y$, where $y_n,y \in Y_0$ the preimage $\pi^{-1}(\{y\} \cup \{y_n\}_{n \in \N})$ is a compact subset of $X$. On the other hand, $\pi^{-1}(\{y\} \cup \{y_n\}) \subset X_0$ by the definition of $X_0$. It follows that the restriction of $\pi$ to $\pi^{-1}(\{y\} \cup \{y_n\})$ is a homeomorphism and $\pi^{-1}(y_n)$ converges to $\pi^{-1}(y)$.
\end{proof}

From Theorem \ref{2811} one can derive the following result which generalizes
the above mentioned result of Huang \cite{H} from Example \ref{e:tameNOThns}.2.

\begin{thm} \label{t:GenHuang} \cite{GM-tame} 
	Let $\pi: X \to Y$ be a strongly almost 1-1 extension 
	of compact metric $G$-systems. 
	Assume that  the dynamical system $(G,Y)$ is tame and that 
	the set
	$p^{-1}(y)$ is (at most) countable for every $p \in E(Y)$ and $y \in Y$,
\footnote{E.g., this latter condition is always satisfied when $Y$ is distal. Another example of such (non-distal) system is the Sturmian system.}
	then $(G,X)$ is also tame.
\end{thm} 
\begin{proof}
We have to show that every $f \in C(X)$ is tame. Assuming the contrary, suppose $fG$
contains an independent sequence $fs_n$.
Since $Y$ is metrizable and tame,
one can assume (by Theorem \ref{D-BFT}) that
the sequence $s_n$  converges pointwise to some element $p$ of $E(G,Y)$.
Consider the set $Y_0 \cap p^{-1}Y_0$, where $Y_0 = \pi(X_0)$. Since $p^{-1}(y)$ is countable for every $y \in Y \setminus Y_0$ it follows that
$Y \setminus (Y_0 \cap p^{-1}Y_0)$ is countable.
Therefore, by the definition of $X_0$ and the countability of $X \setminus X_0$, we see that
$X \setminus \pi^{-1}(Y_0 \cap p^{-1}Y_0)$ is also countable. 
Now observe that the sequence
$(fs_n)(x)$ converges 
for every $x \in \pi^{-1}(Y_0 \cap p^{-1}Y_0)$. 
Indeed if we denote $y=\pi(x)$ then
$s_ny$ converges to $py$ in $Y$. In fact we have $py \in Y_0$ (by the choice of $x$) and $s_n y \in Y_0$. 
By Lemma \ref{claim}, 
$\pi: X_0 \to Y_0$ is a $G$-homeomorphism. So we obtain that $s_n x$ converges
to $\pi^{-1}(py)$ in $X_0$.
Since $f: X \to \R$ is continuous, $(fs_n) (x)$ converges to $f(\pi^{-1}(py))$ in $\R$.
Each $fs_n$ is a continuous function, hence so is also its restriction to
$\pi^{-1}(Y_0 \cap p^{-1}Y_0)$.
Therefore the limit function 
$\phi: \pi^{-1}(Y_0 \cap p^{-1}Y_0) \to \R$ 
is Baire 1. Since C:=$X \setminus \pi^{-1}(Y_0 \cap p^{-1}Y_0)$ 
is countable and $fs_n$ is an independent sequence and
Theorem \ref{2811} provides the sought-after contradiction.
\end{proof}

\begin{cor} \label{c:semic} {\rm (Huang \cite{H} for cascades)}  
	Let $\pi: X \to Y$ be a strong almost 1-1 equivalence 
	 of compact metric $G$-systems,  where $Y$ is equicontinuous. 
	 Then $X$ is tame. 
\end{cor}
%
\begin{proof}
	Consider the induced factor $X \to f(X) \subset Y$ and 
	apply Theorem \ref{t:GenHuang}.  
\end{proof}

Recall the following: 
 \begin{problem} \label{prob:Pisot} (\textit{A version of Pisot conjecture}) \cite[page 31]{diss:Jolivet}
 	Is it true that every (unimodular) irreducible Pisot substitution dynamical system is 
	semi-conjugate to a toral translation~?   
 \end{problem}


 By Corollary \ref{c:semic} a related question is: 
 
 \begin{problem} \label{prob:PisotTame} {\rm (A weaker form of Pisot Conjecture)}  
 	Is it true that every substitutional symbolic dynamical system with Pisot conditions above is always tame ?  
 \end{problem}
 
 \sk 
 
 \begin{remark} \label{r:examples} T. Jolivet \cite[Theorem 3.1.1]{diss:Jolivet}, in the context of Pisot conjecture, discusses some (substitution) dynamical systems which semi-conjugate to a translation on the two-dimensional torus. In particular: 
 	
 	(a) (Rouzy) Tribonacci 3-letter substitution. 
 	
 	(b) Arnoux-Rouzy substitutions. 
 	
 	(c) Brun substitution.
 	
 	(d) Jacobi-Perron substitution.
 	\sk 		
 	By Corollary \ref{c:semic} all these systems are tame. 
 \end{remark}

 \sk 
 \subsection{A special class of generalized Sturmian systems}

 Let $R_{\a}$ be an irrational rotation of the torus $\T^d$. 
 In many cases a reasonably chosen subset $D \subset \T^d$ will yield a generalized Sturmian system.

 \begin{ex} \label{e:spheres}   
 	Let $\a =(\al_1, \dots, \al_d)$ be a vector in $\R^d,\ d \ge 2$
 	with $1,\al_1, \dots, \al_d$ independent over $\Q$.
 	Consider the minimal equicontinuous dynamical system $(R_\al,Y)$,
 	where $Y = \T^d = \R^d / \Z^d$ (the $d$-torus) and $R_\al y = y +\al$.
 	Let $D$ be a small closed $d$-dimensional ball in $\T^d$ and let $C =
 	\partial D$ be its boundary, a  $(d-1)$-sphere.
 	Fix $y_0 \in {\rm{int}} D$ and let
 	$X = X(D,y_0)$ be the symbolic
 	system generated by the function
 	$$
 	x_0 \in \{0,1\}^\Z \ {\text{defined by}}\  x_0(n) = \chi_D(R_\al^ny_0),
 	\qquad
 	X = \overline{\mathcal{O}_{\sig} x_0} \subset  \{0,1\}^\Z,
 	$$
 	where $\sig$ denotes the shift transformation.
 	This is a well known construction and it is not hard to check that
 	the system $(\sig, X)$
 	is minimal and admits $(R_\al,Y)$ as an almost 1-1 factor:
 	$$
 	\pi : (\sig, X) \to (R_\al, Y).
 	$$
 \end{ex}
 
 \begin{thm}\label{sphere}
 	There exists a ball $D \subset \T^d$ as above such that the corresponding symbolic dynamical system 
 	$(\sig, X)$ is tame. For such $D$ we then have a precise description of $E(\sigma,X) \setminus
 	\Z$ as the product set $\T^d \times \mathcal F$, where $\mathcal F$ is the collection of
 	ordered orthonormal bases for $\R^d$.	
 \end{thm}
 
 \begin{proof}
 	{\bf 1.}\
 	First we show that a sphere $C \subset [0,1)^d \cong \T^d$ can be chosen
 	so that for every $y \in \T^d$  the set 
 	$(y + \{n \al : n \in Z\}) \cap C$ is finite.
 	We thank Benjamin Weiss for providing the following proof of this fact.
 	\begin{enumerate}
 		\item
 		For the case $d=2$ the argument is easy.
 		If $A$ is any countable subset of the square
 		$[0,1) \times [0,1)$ there are only a countable
 		number of circles that contain three points of $A$. These circles
 		have some countable collection of radii. Take any circle with a radius
 		which is different from all of them and no translate of it will contain
 		more than two points from the set $A$. Taking $A = \{n\al : n \in \Z\}$
 		we obtain the required circle.
 		\item
 		We next consider the case $d =3$, which easily generalizes to the general case $d \ge 3$.
 		What we have to show is that there can not be infinitely many points in
 		$$
 		A = \{(n\al_1 - [n\al_1],\al_2 - [n\al_2],\al_3 - [n\al_3]): n \in \Z\}
 		$$
 		that lie on a plane.
 		For if that is the case, we consider all $4$-tuples of elements from the set
 		$A$ that do not lie on a plane to get a countable set of radii
 		for the spheres that they determine. 
 		Then taking a sphere with radius different from that collection we obtain our required sphere.
 		In fact, if a sphere contains infinitely many points of
 		$A$ and no $4$-tuple from $A$ determines it then they all lie on a single plane.
 		
 		So suppose that there are infinitely many points in $A$ whose inner
 		product with a vector $v =(z,x,y)$  is always equal to $1$.
 		This means that there are infinitely many
 		equations of the form:
 		\begin{equation}\tag{$*$}
 			z\al_1 + x\al_2 + y\al_3 =
 			1/n + z[n\al_1]/n + x[n\al_2]/n + y[n\al_3]/n.
 		\end{equation}
 		Subtract two such equations with the second using $m$ much bigger than
 		$n$ so that the coefficient of $y$ cannot vanish.
 		We can express $y = rz + sx + t$ with $r, s$  and $t$ rational.
 		This means that we can replace ($*$) by
 		\begin{equation}\tag{$**$}
 			z \al_1 + x\al_2 + y\al_3 =
 			1/n + t[n\al_3]/n + z([n\al_1]/n + r[n\al_3]/n)  +
 			x([n\al_2]/n +s[n\al_3]/n).
 		\end{equation}
 		Now $r, s$ and $t$ have some fixed denominators and (having infinitely
 		many choices) we can take another equation like ($**$) where $n$
 		(and the corresponding $r,s, t$) is replaced by some much bigger $k$,
 		then subtract again to obtain an equation of the form $x = pz + q$ with $p$ and $q$ rational.
 		Finally one more step will show that $z$ itself is rational.
 		However, in view of ($*$), this contradicts the independence of
 		$1, \al_1, \al_2, \al_3$ over $\Q$ and our proof is complete.
 	\end{enumerate}
 	
 	{\bf 2.}\
 	Next we show that for $C$ as above
 	\begin{quote}
 		for every converging sequence $n_i\al$, say $n_i\al \to \beta \in \T^d
 		\cong E(R_\al,\T^d)$, there exists a subsequence $\{n_{i_j}\}$ such
 		that for every $y \in \T^d$, $y + n_{i_j}\al$ is either
 		eventually in the interior of $D$ or eventually in its exterior.
 	\end{quote}
 	Clearly we only need to consider points $y \in C - \beta$.
 	Renaming we can now assume that $n_i \al \to 0$ and that $y \in C$.
 	Passing to a subsequence if necessary we can further assume that
 	the sequence of unit vectors $\frac{n_i\al}{\|n_i \al\|}$ converges,
 	\begin{equation*}
 		\frac{n_i\al}{\|n_i \al\|} \to v_0 \in \mathbb{S}^{d-1}.
 	\end{equation*}
 	In order to simplify the notation we now assume that $C$ is centered
 	at the origin.
 	For every point $y \in C$ where $\langle y, v_0 \rangle \ne 0$
 	we have that $y + n_i\al$ is either eventually in the interior of $D$
 	or eventually in its exterior. On the other hand, for the points $y \in C$
 	with $\langle y, v_0 \rangle = 0$ this is not necessarily the case.
 	In order to deal with these points we need a more detailed information
 	on the convergence of $n_{i}\al$ to $\beta$. At this stage we consider
 	the sequence of orthogonal projections of the vectors $n_{i}\al$ onto
 	the subspace $V_1 = \{u \in \R^d : \langle u , v_0 \rangle =0\}$, say
 	$u_i= \proj_{v_0}(n_i \al) \to u = \proj_{v_0}(\beta)$.
 	If it happens that eventually $u_i =0$, this means that all
 	but a finite number of the $n_i \al$'s are on the line defined by $v_0$
 	and our required property is certainly satisfied
 	\footnote{Actually this possibility can not occur,
 		as is shown in the first step of the proof.}.
 	Otherwise
 	we choose a subsequence (again using the same index) so that
 	\begin{equation*}
 		\frac{u_i}{\|u_i\|} \to v_1 \in \mathbb{S}^{d-2}.
 	\end{equation*}
 	Again (as we will soon explain) it is not hard to see that for points
 	$y \in C \cap V_1$ with
 	$\langle y, v_1 \rangle \ne 0$ we have that $y + n_{i}\al$ is either eventually in the interior of $D$ or eventually in its exterior.
 	For points $y \in C \cap V_1$ with $\langle y, v \rangle = 0$ we have
 	to repeat this procedure. Considering the subspace
 	$V_2 = \{u \in V_1 : \langle u , v_1 \rangle =0\}$,
 	we define the sequence of projections $u'_i= \proj_{v_1}(u_i) \in V_2$
 	and pass to a further
 	subsequence which converges to a vector $v_2$
 	\begin{equation*}\label{dir}
 		\frac{u'_i}{\|u'_i\|} \to v_2 \in \mathbb{S}^{d-3}.
 	\end{equation*}
 	Inductively this procedure will produce an {\bf ordered orthonormal basis}
 	$\{v_0, v_1, \dots,v_{d-1}\}$ for $\R^d$ and a final
 	subsequence (which for simplicity we still denote
 	as $n_i$) such that
 	
 	\begin{quote}
 		for each $y \in \T^d$,
 		$y + n_i\al$ is either eventually in the \\
 		interior of $D$
 		or it is eventually in its exterior.
 	\end{quote}
 	
 	This is clear for points $y \in \T^d$ such that $y + \beta \not\in C$.
 	Now suppose we are given a point $y$ with $y + \beta \in C$.
 	We let $k$ be the first index with
 	$\langle y + \beta, v_k \rangle \ne 0$. As $\{v_0, v_1, v_2, \dots,v_{d-1}\}$
 	is a basis for $\R^d$ such $k$ exists.
 	We claim that
 	the sequence $y + n_i \al$ is either eventually in the interior of $D$ or it is eventually in its exterior.
 	To see this consider the affine hyperplane
 	which is tangent to $C$ at $y +\beta$
 	(which contains the vectors $\{v_0,\dots, v_{k-1}\}$).
 	Our assumption implies that the sequence $y + n_i \al$ is either
 	eventually on the opposite side of this 
 	hyperplane from the sphere, in which
 	case it certainly lies in the exterior of $D$, or it eventually lies on the same side as the sphere.
 	However in this latter case it can not be squeezed
 	in between the sphere and the tangent hyperplane, as this would imply
 	$\langle y + \beta, v_k \rangle = 0$, contradicting our assumption.
 	Thus it follows that in this case the sequence
 	$y + n_i \al$ is eventually in the interior of $D$.

 	{\bf 3.}\
 	Let now $p$ be an element of $E(\sig, X)$.
 	We choose a {\bf net} $\{n_\nu\} \subset \Z$ with $\sig^{n_\nu} \to p$.
 	It defines uniquely an element $\beta \in E(Y) \cong \T^d$ so that
 	$\pi(px) = \pi(x) + \beta$ for every $x \in X$.
 	Taking a subnet if necessary we can assume
 	that the net $\frac{\beta - n_\nu\al}{\|\beta - n_\nu \al\|}$ converges  to some $v_0 \in S^{d-1}$. And, as above, proceeding by induction we assume likewise that all the corresponding limits $\{v_0,\dots, v_{k-1}\}$ exist.
 	
 	Next we choose a {\bf sequence} $\{n_i\}$ such that
 	$n_i \al \to \beta$,
 	$\frac{\beta - n_i\al}{\|\beta - n_i \al\|} \to v_0$ etc.
 	We conclude that $\sig^{n_i} \to p$.
 	Thus every element of $E (\sig,X)$ is obtained as
 	a limit of a sequence in $\Z$ and is therefore of Baire class 1.
 	
 	{\bf 4.}\	
 	From the proof we see that the elements of $E(\sigma,X) \setminus
 	\Z$ can be parametrized by the set
 	$\T^d \times \mathcal F$, where $\mathcal F$ is the collection of
 	ordered orthonormal bases for $\R^d$,
 	$\ p \mapsto (\beta, \{v_0,\dots, v_{d-1}\})$.
 \end{proof}

 \sk 
\section{Remarks about order preserving systems}
\label{s:order}

\subsection{Order preserving action on the unit interval}
\label{s:orderOnI}

Recall that for the group $G=\H_+[0,1]$ comprising the orientation preserving self-homeomorphisms
of the unit interval,
the $G$-system $X=[0,1]$  with the obvious $G$-action is tame \cite{GM-AffComp}.
One way to see this is to observe that the enveloping semigroup of this dynamical
system naturally embeds into the Helly compact space (and hence is a Rosenthal compact space).
By Theorem \ref{D-BFT}, $(G,X)$ is tame. 
We list here some other properties of $\H_+[0,1]$.

\begin{remark} \label{H_+[0,1]}
Let $G:=\H_+[0,1]$. Then
\ben
\item  \emph{(Pestov \cite{Pest98})} \ $G$ is extremely amenable.
\item \cite{GM-suc} \ $\WAP(G)=\Asp(G)=\SUC(G)=\{constants\}$ and every Asplund  representation of $G$ is trivial.
\item \cite{GM-AffComp} \ $G$ is representable on a (separable) Rosenthal space. 
\item
\emph{(Uspenskij \cite[Example 4.4]{UspComp})}
\ $G$ is Roelcke precompact.
\item
$\UC(G) \subset \Tame(G)$,
that is, the Roelcke compactification of $G$ is tame.
\item $\Tame(G) \neq \UC (G)$.
\item $\Tame(G) \neq \RUC(G)$, that is, $G$ admits a transitive dynamical system which is not tame.
\item  \cite{MePolev}  \ $\H_+[0,1]$ and $\H_+(\T)$ are minimal topological groups.
\een
\end{remark}

In properties (5) and (6) we answer two questions of T. Ibarlucia 
which are related to  \cite{Ibar}.
For the details see \cite{GM-tame}.  

\begin{thm} \label{t:RP}
The Polish group $G=\H_+(\T)$ is Roelcke precompact.
\end{thm}
\begin{proof}
First a general fact: if a topological group $G$ can be represented as
$G=KH$,
where $K$ is a compact subset and $H$ a Roelcke-precompact subgroup, then $G$ is also Roelcke-precompact.
This is easy to verify either directly
or by applying \cite[Prop. 9.17]{RD}.
As was mentioned in Theorem \ref{H_+[0,1]}.4,  $\H_+[0,1]$ is Roelcke precompact.
Now, observe that in our case
$G = KH$, where $H:=St(1) \cong \H_+[0,1]$ is the stability group of
$1 \in \T$ and $K \cong \T$ is the subgroup
of $G$ consisting of the rotations of the circle.
Indeed, the coset space $G/H$ is homeomorphic to $\T$ and there exists a natural continuous section $s: \T \to K \subset G$.
\end{proof}

\subsection{Circularly ordered systems} 

In \cite{GM-c} we introduce the class of circularly ordered (c-ordered) dynamical systems which naturally generalizes the class of linearly ordered systems. A compact $S$-system $X$ is said to be c-ordered 
(notation $(S,X) \in $ CODS) 
if the topological space $X$ is c-ordered and every $s$-translation $X \to X$ is c-order preserving.

 \begin{ex} \label{ex:c} \ 
 	\ben 
 	\item With every c-ordered compact space there is the associated topological group $\H_+(X)$ of c-order preserving homeomorphisms. Certainly, $X$ is a c-ordered $\H_+(X)$-system. 
	Every linearly ordered $G$-system is c-ordered. 
 	\item 
 	The Sturmian like $\Z^k$-subshifts 
 	(see Theorem \ref{multi}) 
 	admit a circular order. Moreover, their enveloping semigroups also are c-ordered systems, \cite{GM-c}. 
 	\item 
 	Every element $g$ of the projective group $\PGL(2,\R)$ defines a homeomorphism on the circle $\T \to \T$ which 
	is either c-order preserving or $c$-order reversing. 
 	\een
 \end{ex}

\begin{thm} \label{t:CoisWRN} \cite{GM-c} 
	\ben 
	\item 
	Every c-ordered compact, not necessarily metrizable, $S$-space $X$ is Rosenthal representable 
	(that is, $\mathrm{WRN}$), hence, in particular, tame. 
	So, $\mathrm{CODS} \subset \mathrm{WRN} \subset \mathrm{Tame}$. 
		\item \label{t:GrRepr}  
		The topological group $\H_+(X)$ (with compact open topology) is 
		Rosenthal representable for every c-ordered compact space $X$. 
		For example, this is the case for $\H_+(\T)$.  
		\een
\end{thm}

The Ellis compactification $j: G \to E(G,\T)$ of the group $G=\H_+(\T)$ is a topological embedding.
In fact, observe that
the compact open topology on $j(G) \subset C_+(\T, \T)$ coincides
with the pointwise topology.
This observation implies, by \cite[Remark 4.14]{GM-survey}, that
$\Tame(G)$ separates points and closed subsets. For any group $G$ having sufficiently many tame functions the universal tame semigroup compactification $G \to G^{\Tame}$ is a topological embedding.

\begin{remarks} \ 
	\ben 
	\item Regarding Theorem \ref{H_+[0,1]}.2 we note that recently Ben-Yaacov and Tsankov \cite{BenTsankov}
	found some other Polish groups $G$ for which $\WAP(G) = \{constants\}$ (and which are therefore also reflexively trivial).
	
	\item Although $G=\H_+(\T)$ is representable on a (separable) Rosenthal Banach space, 
	the group $\H_+(\T)$ is Asplund-trivial. Indeed, it is algebraically simple
	\cite[Theorem 4.3]{Ghys} and 
	contains a copy of $\H_+[0,1]=St(z)$ (a stabilizer group of some point $z \in \T$) which is Asplund-trivial \cite{GM-suc}.
	Now, as in \cite[Lemma 10.2]{GM-suc},
	use an observation of Pestov, which implies that 
	any continuous Asplund representation of $\H_+(\T)$ is trivial. 
	\een
\end{remarks}

\begin{question}  Is it true that the universal tame compactification 
	$u_t: G \to G^{\Tame}$ is an embedding for every Polish group $G$ ? 
\end{question}

The universal Polish group $G=H[0,1]^{\N}$ is a natural candidate for a counterexample. 


\sk 
	\subsection{Noncommutative Sturmian like symbolic systems} 
	\label{s:nonc} 
		
%
	The following construction yields many tame coding functions for subgroups of $\H_+(\T)$,
	and via any abstract  homomorphism $h: G \to \H_+(\T)$,	 
	we obtain coding functions on $G$.
%
	
	 \begin{thm} \label{t:multi} 
	 		Let $h: G \to \H_+(\T)$ be a group homomorphism and let
	 	$$
	 	f: \T \to A:=\{0, \dots ,d\}
	 	$$
	 	be a finite coloring map induced by a finite partition of the circle $\T$ comprising disjoint arcs. 
	  
	 	Then, for any given point $z \in \T$ we have:
	 	\ben  	
	 \item the coding function $\varphi=m(f,z): G \to \{0, \dots ,d\}$ is tame on the discrete copy of $G$. 
	 \item the corresponding symbolic $G$-system $G_{\varphi} \subset  \{0,1, \cdots, d\}^{G}$ is tame. 
	 \item  if the action of $G$ on $\T$ is minimal then, in many cases, the $G$-system $G_{\varphi}$ is minimal and circularly ordered.
	 \een
	 \end{thm}
	 \begin{proof}(A sketch)
	 The coloring map $f: \T \to A:=\{0, \dots ,d\}$ has bounded variation. Every circle homeomorphism $g \in \H_+(\T)$ is 
	 circular order preserving. 
	 This implies that 
	 the orbit $fG=\{fg: g \in G\}$, as a bounded family of real (discontinuous) functions on $\T$, 
	 has bounded total variation.  	
	As we know by \cite{Me-Helly, GM-c} any such family on $\T$ (or, on any other circularly ordered set)  is tame. 
	From Lemma \ref{l:tametype}.4 we conclude that $\varphi=m(f,z)$ is a tame function. This yields (1) and (2). 
	For (3) we use some results from \cite{GM-c}. 
	\end{proof} 
	
	\begin{remark} \label{r:nonc}  
	Some particular cases of this construction (for a suitable $h$) are as follows:  
	\ben 
	\item Sturmian and Sturmian like multidimensional symbolic $\Z^k$-systems. 
	\item Consider a subgroup $G$ of $\PSL_2(\R)$, isomorphic to $F_2$,
	which is generated by two M\"{o}bius transformations as in \cite{GM-fp}, say an irrational rotation and 
	a parabolic transformation. 
	When $d=1$ (two colors) 
	we get the corresponding minimal tame subshift $X \subset \{0,1\}^{F_2}$. 
	\item 
One can consider coding functions on any Fuchsian group $G$.	
E.g., for the noncommutative modular group $G=\PSL_2(\Z) \simeq \Z_2 \ast \Z_3$. 
	\item More generally, at least in the assertions (1) and (2) of Theorem \ref{t:multi}, one may replace 
	the circle $\T$ by any circularly ordered set $X$ with circular order preserving $f: X \to A$. 
	\een 
	\end{remark}

	 \sk 
\section{Tame minimal systems and topological groups} 

\label{sec,int}

Recall that for every topological group $G$ there exists a
unique universal minimal $G$-system $M(G)$.
Frequently 
 $M(G)$ is nonmetrizable. For example, this is the case for every locally compact noncompact $G$. On the other hand,
many interesting massive Polish groups are extremely amenable that is, having trivial $M(G)$.
See for example \cite{Pest98, PestBook, UspComp,van-the2}.
The first example of a nontrivial yet metrizable $M(G)$ was found by Pestov.
In \cite{Pest98} he shows that for $G:=\H_+(\T)$ the universal minimal system $M(G)$
can be identified with the natural action of $G$ on the circle $\T$.
Glasner and Weiss \cite{GWsym,GW-Cantor} gave an explicite description
of $M(G)$ for the symmetric group $S_{\infty}$ and for $\H(C)$
(the Polish group of homeomorphisms of the Cantor set $C$).
Using model theory Kechris, Pestov and Todor\u{c}evi\'{c} gave in \cite{KPT} many new
examples of various subgroups of $S_{\infty}$ with metrizable (and computable) $M(G)$.

\sk
 
Note that the universal almost periodic factor $M_{AP}(G)$ of $M(G)$ is the Bohr compactification $b(G)$ of $G$. 
When the induced homomorphism $G \to \Homeo b(G)$ is injective (trivial) 
the topological group $G$ is called maximally (resp., minimally) almost periodic.  
Every topological group $G$ has a universal minimal tame system 
$M_t(G)$ which is the largest tame $G$-factor of $M(G)$. It is not necessarily AP (in contrast to the HNS and WAP cases). 
 There are (even discrete) minimally almost periodic groups which however 
 admit effective minimal tame systems, or 
in other words, groups for which
 the corresponding homomorphism $G \to \Homeo (M_t(G))$ is injective. 
 For example, the countable group $\PSL_2(\Q)$ is minimally almost periodic (von Neumann and Wigner) 
 even in its discrete topology.  
It embeds densely into the group $\PSL_2(\R)$ which  
acts effectively and transitively on the circle. 
 Thus, the circle provides a topologically effective minimal action for every dense subgroup $G$ of $\PSL_2(\R)$. 
 In particular, it is effective (though not topologically effective) for the discrete copy of $\PSL_2(\Q)$). 

\begin{question} \label{q:tMIN} 
	Which Polish groups (e.g., discrete countable groups) $G$ have effective tame minimal actions~?
	Equivalently, when is the homomorphism $G \to \Homeo (M_t(G))$ injective? 

	
\end{question}


Next we will discuss in more details the question ``when is $M(G)$ tame ?".

\begin{defin} \label{d:int-tame} 
	We say that a topological group $G$ is \emph{intrinsically tame} 
	if one of the following equivalent conditions is satisfied:
	
	\ben 
		\item every continuous action of $G$ on a compact space $X$ admits a $G$-subsystem $Y \subset X$ which is tame. 
		\item any minimal compact $G$-system is tame.
	\item the universal minimal $G$-system $M(G)$ is tame.
	\item the natural projection $M(G) \to M_t(G)$ is an isomorphism. 
	\een 
\end{defin}
 

The $G$-space $M_t(G)$ can also be described as a minimal left ideal in the 
universal tame $G$-system $G^{\Tame}$. 
Recall that $G \to G^{\Tame}$ is a semigroup $G$-compactification 
determined by the algebra $\Tame(G)$. 
The latter is isomorphic to its own enveloping semigroup
and thus has a structure of a compact right topological semigroup. Moreover, any
two minimal left ideals there, are isomorphic as dynamical systems.

In \cite{GM1} we defined, for a topological group $G$ and
a dynamical property $P$, the notion of $P$-fpp
($P$ fixed point property). Namely $G$ has the $P$-fpp if every $G$-system which
has the property $P$ admits a $G$ fixed point.
Clearly this is the same as demanding that every minimal $G$-system with the property
$P$ be trivial. Thus for $P=\Tame$ a group $G$ has the tame-fpp iff $M_t(G)$ is trivial.


\sk 
We will need the following
theorem which extends a result in \cite{Gl-str}.

\begin{thm}\label{pd}
Let $(G,X)$ be a metrizable minimal tame dynamical system and suppose
it admits an invariant probability measure. Then $(G,X)$ is point distal.
If moreover, with respect to $\mu$ the system $(G,\mu,X)$ is weakly
mixing then it is a trivial one point system.
\end{thm}

\begin{proof}
With notations as in \cite{Gl-str} we observe that for any minimal
idempotent $v \in E(G,X)$ the set $C_v$ of continuity points of $v$
restricted to the set $\ov{vX}$, is a dense $G_\del$ subset of $\ov{vX}$
and moreover $C_v \subset vX$ (\cite[Lemma 4.2.(ii)]{Gl-str}).
Also, by \cite[Proposition 4.3]{Gl-str} we have $\mu(vX)=1$, and it follows
that $\ov{vX}=X$. The proof of the claim that $(G,X)$ is point distal
is now finished as in \cite[Proposition 4.4]{Gl-str}.

Finally, if 
the measure preserving system $(G,\mu,X)$ is weakly
mixing it follows that it is also topologically weakly mixing. By the
Veech-Ellis structure theorem for point distal systems
\cite{V, E}, if $(G,X)$ is nontrivial it admits a nontrivial equicontinuous
factor, say $(G,Y)$. However $(G,Y)$, being a factor of $(G,X)$, is
at the same time also topologically weakly mixing which is a contradiction.
\end{proof}



\begin{thm} \label{t:intrinsic} \
\begin{enumerate}
\item
Every extremely amenable group is intrinsically tame.
\item
The Polish group $\H_+(\T)$ of orientation preserving homeomorphisms of the circle is intrinsically tame.
\item
The Polish groups $\Aut(\mathbf{S}(2))$ and 
$\Aut(\mathbf{S}(3))$, of automorphisms 
of the circular directed graphs $\mathbf{S}(2)$ and $\mathbf{S}(3)$, are intrinsically tame.
\item
A discrete group which is intrinsically tame is finite. 
\item
For an abelian infinite countable discrete group $G$, its universal
minimal tame system $M_t(G)$ is a highly proximal extension of
its Bohr compactification $G^{AP}$ (see e.g. \cite{Gl-str}). 
\item
The Polish group $\H(C)$, of homeomorphisms of the Cantor set,
is not intrinsically tame.
\item
The Polish group $G=S_{\infty}$, of permutations of the natural numbers,
is not intrinsically tame. In fact $M_t(G)$ is trivial; i.e. $G$ has the tame-fpp.
\end{enumerate}
\end{thm}

\begin{proof}
(1) Is trivial. 

(2) Follows from Pestov's theorem \cite{Pest98}, which
identifies $(G,M(G))$ for $G = \H_+(\T)$ as the tautological action of $G$ on $\T$, and from Theorem \ref{t:CoisWRN} which asserts that this 
system is tame (being c-ordered). 

(3) 
The universal minimal $G$-systems for the groups $\Aut(\mathbf{S}(2))$ and 
$\Aut(\mathbf{S}(3))$ are computed in \cite{van-the}.
In both cases it is easy to check that every element of the enveloping semigroup
$E(M(G))$ is an order preserving map. As there are only $2^{\aleph_0}$ order
preserving maps, it follows that the cardinality of $E(M(G))$ is
$2^{\aleph_0}$, whence, in both cases, the dynamical system $(G, M(G))$ is tame.

In order to prove Claim (4) we assume, to the contrary, that $G$ is infinite
and apply a result of B. Weiss  \cite{W}, to obtain
a minimal model, say $(G,X,\mu)$, of the
Bernoulli probability measure preserving
system $(G,\{0,1\}^G,  (\frac12 (\del_0 + \del_1))^G)$.
Now $(G,X,\mu)$ is metrizable, minimal and tame, and it carries
a $G$-invariant probability measure with respect to which the system is weakly mixing.
Applying Theorem \ref{pd} we conclude that $X$ is trivial.
This contradiction finishes the proof. 
\footnote{Modulo an extension of Weiss' theorem, which does not yet exist, 
a similar idea would work for any locally compact group.
The more general statement would be: A locally compact group which is
intrinsically tame is compact.}

(5) 
In \cite{H}, \cite{KL} and \cite{Gl-str} it is shown that a metric minimal tame $G$-system
is an almost one-to-one extension of an equicontinuous system.
Now tameness is preserved under sub-products, and because our group $G$ is
countable, it follows that $M_t(G)$ is a minimal sub-product of all the
minimal tame metrizable systems. In turn this implies that $M_t(G)$ is
a (non-metrizable)  highly proximal extension of the Bohr compactification
$G^{AP}$ of $G$.

(6) 
To see that $G = \H(C)$ is not intrinsically tame 
it suffices to show that the tautological action $(G,C)$, which is a factor of $M(G)$, 
is not tame. To that end note that the shift transformation
$\sig$ on $X = \{0,1\}^\Z$ is a homeomorphism of the
Cantor set. Now the enveloping semigroup $E(\sig,X)$ of the cascade
$(\sig,X)$, a subset of $E(G,X)$, is homeomorphic to $\beta\N$.

(7) 
To see that $G = S_{\infty}$ is not intrinsically tame we recall first that,
by \cite{GW}, the universal minimal dynamical system for this group
can be identified with the natural action of $G$ on the compact metric
space $X = LO(\N)$ of linear orders on $\N$. Also, it follows from the
analysis of this dynamical system that for any minimal idempotent
$u \in E(G,X)$ the image of $u$ contains exactly two points,
say $uX = \{x_1,x_2\}$.
A final fact that we will need concerning the system $(G,X)$ is that
it carries a $G$-invariant probability measure $\mu$ of full support \cite{GW}.
Now to finish the proof, suppose that $(G,X)$ is tame. Then there
is a {\bf sequence} $g_n \in G$ such that $g_n \to u$ in $E(G,X)$.
If $f \in C(X)$ is any continuos real valued function, then we have,
for each $x \in X$,
$$
\lim_{n \to \infty} f(g_n x) = f(ux) \in \{f(x_1), f(x_2)\}.
$$
But then, choosing a function $f \in C(X)$ which vanishes 
at the points\ $x_1$ and $x_2$ and with $\int f\,d\mu=1$, we get, 
by Lebesgue's theorem,
$$
1 = \int f \, d\mu = \lim_{n\to \infty} \int f(g_nx) \, d\mu =
\int f (ux) \, d\mu =0.
$$

Finally, the property of supporting an invariant measure, as well as the 
fact that the cardinality of the range of minimal idempotents is
$\le 2$, are inherited by factors and thus the same argument shows that $M(G)$
admits no nontrivial tame factor.
Thus $M_t(G)$ is trivial.  
\end{proof}

We will say that $G$ is \emph{intrinsically c-ordered} 
if the $G$-system $M(G)$ is circularly ordered.  
Using this terminology Theorem \ref{t:intrinsic} says that the Polish groups $G=\H_+(\T)$, $\Aut(\mathbf{S}(2))$ and 
$\Aut(\mathbf{S}(3))$ are intrinsically c-ordered. 
Note that for $G=\H_+(\T)$ every compact $G$-space $X$ contains a copy of $\T$ as  $G$-subspace or a $G$-fixed point. 

\vspace{0.3cm}

The (nonamenable) group $G=\H_+(\T)$ has one more remarkable property. Besides $M(G)$,
one can also effectively compute the affine analogue of $M(G)$.
Namely, the \emph{universal irreducible affine system} of $G$ (we denote it by $I\!A(G)$) which was defined
and studied in \cite{Gl-thesis, Gl-book1}.
It is uniquely determined up to affine isomorphisms.  
For any topological group $G$ 
the corresponding affine compactification $G \to I\!A(G)$ coincides with the
affine compactification $G \to P(M_{sp}(G))$, where, $M_{sp}(G)$ is the 
\emph{universal strongly proximal minimal system} of $G$
and $P(M_{sp}(G))$ is the space of probability measures on the compact
space $M_{sp}(G)$. 
For more information regarding affine compactifications of dynamical systems we refer to \cite{GM-AffComp}. 

\begin{defin} \label{d:ConvIntTame}
We say that $G$ is \emph{convexly intrinsically tame} 
if one of the following equivalent conditions is satisfied:

\ben 
\item every compact affine dynamical system
$(G,Q)$ admits an affine tame $G$-subsystem. 
\item every compact affine dynamical system
$(G,Q)$ admits a tame $G$-subsystem. 
\item every irreducible affine $G$-system is tame.
\item the universal irreducible affine $G$-system $I\!A(G)$ is tame.  
\een 
\end{defin}

Note that the $G$-system $P(X)$ is affinely universal for a $G$-system $X$; also, 
$P(X)$ is tame whenever $X$ is \cite[Theorem 6.11]{GM-rose}, \cite{GM-AffComp}. 
In particular, it follows that any intrinsically tame group is convexly intrinsically tame.

It is well known that a topological group $G$ is amenable iff $M_{sp}(G)$ is trivial
(see \cite{Gl-book1}). Thus
$G$ is amenable iff $I\!A(G)$ is trivial and it follows that every amenable group is trivially convexly intrinsically tame.

 Thus we have the following diagram which
emphasizes the analogy between the two pairs of properties:
\begin{equation*}
\xymatrix
{
 \text{extreme amenability}\  \ar@2{->} [d] \ar@2{->}[r]\  &  \  \text{intrinsically tame} \ar@2{->}[d]\  \\
 \text{amenability}\  \ar@2{->}[r] &  \  \text{convexly intrinsically tame}
}
\end{equation*}

\begin{remark} \label{r:collapsing} 
Given a class $P$ of compact $G$-systems which is stable under subdirect products, one 
can define the notions of an intrinsically $P$ group and a convexly intrinsically 
$P$ group in a manner analogous to the one we adopted for $P=\Tame$.
We then note that 
in this terminology a group is convexly intrinsically HNS (and, hence, also conv-int-WAP) iff it is amenable. 
This follows easily from the fact that the algebra $\Asp(G)$ is left amenable, \cite{GM-fp}. 
This ``collapsing effect" together with the special role of tameness in 
the dynamical BFT dichotomy \ref{D-BFT} suggest that the 
notion of convex intrinsic tameness is a natural analogue of amenability. 
\end{remark}

At least for discrete groups, if $G$ is intrinsically HNS then it is finite.
In fact, for any group, an HNS minimal system is equicontinuous 
(see \cite{GM1}), so that
for a group $G$ which is intrinsically HNS the universal minimal system $M(G)$ coincides with its Bohr compactification $G^{AP}$.
Now for a discrete group, it is not hard to show that an infinite minimal equicontinuous 
system admits a nontrivial almost one to one (hence proximal) 
extension which is still minimal. 
Thus $M(G)$ must be finite.
However, by a theorem of Ellis \cite{Ellis}, for discrete groups the group $G$ acts freely on $M(G)$,
so that $G$ must be finite as claimed. Probably similar arguments will show that 
a locally compact intrinsically HNS group is necessarily compact.

\begin{thm} \label{t:conv-int-tame} \
	\begin{enumerate}
		\item $S_{\infty}$
		is amenable (hence convexly intrinsically tame) but not intrinsically tame. 
		\item
		$\H(C)$ 
		is not convexly intrinsically tame.
		\item $\H([0,1]^{\N})$ is not convexly intrinsically tame.
		
		\item $\H_+(\T)$ 
		is a (convexly) intrinsically tame nonamenable topological group. 
		
		\item $SL_n(\R)$, $n >1$ 
		(more generally, any semisimple Lie group $G$ with finite center and no compact factors) is convexly intrinsically tame nonamenable topological group. 
		
	\end{enumerate}
\end{thm}
\begin{proof}
	(1) $S_{\infty}$ is  amenable, \cite{Harpe}.  
	It is not intrinsically tame by Theorem \ref{t:intrinsic}.7.
	
	(2) 
	Natural action of $\H(C)$ on the Cantor set $C$ is minimal and strongly proximal, 
	but this action is not tame; it contains, as a subaction, a copy of the full shift $(\Z,C) \cong (\sigma,\{0,1\}^\Z)$.
	
	(3)  The group $\H([0,1]^{\N})$ is a universal Polish group (see Uspenskij \cite{UspUn}).  
	It is not convexly intrinsically tame.
	This can be established by observing that the action of this group on the Hilbert cube is minimal, 
	strongly proximal and not tame.
	The strong proximality of this action can be
	easily checked.
	The action is not tame because it is a \emph{universal action} (see \cite{MeNZ}) 
	for all Polish groups on compact metrizable spaces.
	
	(4) 
	The (universal) minimal $G$-system $\T$ for $G=\H_+(\T)$ is strongly proximal. 
	Hence, $I\!A(G)$ in this case is easily computable
	and it is exactly $P(\T)$ which, as a $G$-system, is tame (by Theorem \ref{t:intrinsic}.4). 
	Thus, $\H_+(\T)$ is a (convexly) intrinsically tame. 

(5) 
By Furstenberg's result \cite{Furst-63-Poisson} 
the universal minimal strongly proximal system
$M_{sp}(G)$ is the homogeneous space $X=G/P$, where $P$ is a minimal parabolic 
subgroup (see \cite{Gl-book1}).
Results of Ellis \cite{Ellis93} and Akin \cite{Ak-98} (Example \ref{e:tameNOThns}.1) 
show that the enveloping semigroup $E(G,X)$ in this case is a Rosenthal compact space, 
whence the  system $(G,X)$ is tame by the dynamical BFT dichotomy (Theorem \ref{D-BFT}). 	
	\end{proof}

In particular, for $G=SL_2(\R)$ note that 
in any compact \textit{affine} $G$-space we can find either a
	1-dimensional real projective $G$-space
	(a copy of the circle) or a fixed point.  
For general $SL_n(\R)$, $n \geq 2$ -- flag manifolds and their $G$-quotients. 


\bibliographystyle{amsplain}

\end{document}